\documentclass[preprint,11pt]{elsarticle}

\usepackage{amssymb}
\usepackage{amsmath}
\usepackage{graphicx}

\usepackage{amsthm}

\usepackage{mathrsfs}

\usepackage{natbib}
\usepackage{multirow}%
\usepackage{adjustbox}
\usepackage{subfig} 
\usepackage{xspace}
\usepackage{float}
\usepackage[framemethod=tikz]{mdframed}
\usepackage{subfig} %
\usepackage{caption}   
\usepackage{enumitem}
\usepackage[utf8]{inputenc}
\usepackage{booktabs}  
\usepackage{tabularx}

\newtheorem{theorem}{Theorem}
\newtheorem{proposition}[theorem]{Proposition}%

\newtheorem{assumption}[theorem]{Assumption}

\newtheorem{example}{Example}%

\begin{document}

\begin{frontmatter}



\title{Conditional Expectation expression in mean-field SDEs and its applications}


\author{Samaneh Sojudi and Mahdieh Tahmasebi} 

\affiliation{Department of Applied Mathematics, Faculty of Mathematical Sciences, Tarbiat Modares University, P.O. Box 14115-134, Tehran,
Iran}

\begin{abstract}
This study developed a novel formulation of conditional expectations within the framework of a jump-diffusion mean-field stochastic differential equation. We introduce an integrated approach that combines unconditioned expectations with rigorously defined weighting factors, employing Malliavin calculus on Poisson space and directional derivatives to enhance estimation accuracy.
\noindent The proposed method is applied to the numerical pricing of American put options in a jump-diffusion mean-field setting, addressing the challenges proposed by early-exercise features. Comprehensive numerical experiments demonstrate substantial improvements in pricing accuracy compared with conventional techniques.
\noindent Moreover, we investigate variance-reduction techniques with a focus on localization and control variates, and discuss their practical implementation in detail.
\end{abstract}



\begin{keyword}
Malliavin Calculus, mean-field, jump diffusion, Poisson space, American options
\end{keyword}

\end{frontmatter}




\section{Introduction}

\noindent In systems where many agents or market participants interact, the individual dynamics of asset prices or risk factors intertwine with the aggregate behavior of the entire system. This “mean-field” influence plays a crucial role in modeling assets within interconnected markets. In such scenarios, agents adjust their strategies based on the statistical distribution of all participants' states rather than on isolated individual dynamics. For instance, in energy markets or smart grid systems, the price of an energy commodity is affected not only by macroeconomic factors but also by the collective consumption and production behavior of many participants. Extended mean-field game models capture these interactions by incorporating coupling terms that depend on the law of the underlying state variable. Mean-field stochastic differential equations (SDEs)—also known as McKean-Vlasov SDEs or distribution-dependent SDEs—trace their origins to the pioneering work of Kac \citep{ref-kac1956} in 1956, which was subsequently extended by McKean \citep{ref-mckean1966} in 1966 in the study of large systems of interacting stochastic particles. Subsequent advances by Graham \citep{ref-graham1992} introduced McKean-Vlasov Itô-Skorohod equations and nonlinear diffusions with discrete jump sets. These equations provide a rigorous framework for modeling the collective dynamics of large ensembles, making them indispensable in both statistical physics and the quantitative analysis of large-scale social and economic interactions. \\
McKean-Vlasov SDEs have become central to contemporary research on mean-field games and the pricing of complex financial derivatives, as discussed extensively in \citep{ref-carmona2018} and \citep{ref-bush2011}, \citep{ref-hambly2019}, \citep{ref-dawson1995}, and the references therein. The integration of mean-field effects with jump components has recently emerged as a powerful modeling paradigm, capturing both the abrupt shocks typical of financial markets and the collective behavior inherent in complex systems. In these frameworks, the evolution of asset prices or risk factors becomes inseparable from the statistical distribution of the entire system, an aspect central to both financial markets and energy networks. The extension of mean-field games to include jump processes—such as those governed by doubly Poisson systems—allows for realistic analyses of phenomena ranging from sudden market crashes to strategic shifts in energy demand. Motivated by the effectiveness of such models, this work investigates the pricing of American options whose the underlying asset dynamics are governed by a mean-field stochastic differential equation with jumps.\\
\noindent Pricing and hedging American options in high dimensions remains a persistent and significant challenge in computational finance. Traditional deterministic techniques—including finite difference, finite element, and lattice schemes—are rendered computationally impractical due to the exponential growth in cost and memory requirements associated with the curse of dimensionality, especially when $d>3$. Consequently, Monte Carlo methods emerge as the only feasible computational approach for accurately estimating option prices and sensitivities in these high-dimensional settings.\\
\noindent Recent advancements have coalesced around three primary methodological streams for tracking this problem. The first stream employs tree-based grid discretizations of diffusion processes, typified by the fundamental work of Broadie and Glasserman \citep{ref-broadie1997}, Bally, Pagès, and Printems \citep{ref-bally2003}, and Barraquand and Martineau \citep{ref-barraquand1995}. The second stream applies truncated-basis regression in $\mathcal{L}^2$ spaces to estimate the crucial conditional expectations, pioneered by Longstaff and Schwartz \citep{ref-longstaff2001} and further developed by Tsitsiklis and Van Roy \citep{ref-tsitsiklis2002}. The third and the most relevant stream utilizes Malliavin calculus to derive exact representation formulas, enabling purely Monte Carlo algorithms even in complex, nonlinear frameworks, as demonstrated in \citep{ref-bouchard2004}, \citep{ref-fournie1999}, and \citep{ref-fournie2001}.\\
\noindent The application of Malliavin calculus to solve the complexities of American option pricing under stochastic dynamics is a fertile and rapidly expanding area of research \citep{ref-abbas2012, ref-bally2005, ref-kharrat2019, ref-kharrat2022}. Among the notable advancements are the representation formulas for constant volatility models introduced by Bally et al. \citep{ref-bally2005}, and subsequent generalizations by Abbas-Turki and Lapeyre \citep{ref-abbas2012}, Kharrat \citep{ref-kharrat2019}, and Kharrat and Bastin \citep{ref-kharrat2022}, which extend these techniques to various settings involving stochastic volatility. Petrou further broadened the theoretical framework by generalizing the results of Fournié et al. to financial markets influenced by Lévy processes. More specialized applications include Mancino's novel weight technique for Delta hedging under local volatility \citep{ref-mancino2020}, and Saporito's advancement of multiscale stochastic volatility models for the valuation of path-dependent derivatives \citep{ref-saporito2020}, and Yamada's approximation scheme for multidimensional Stratonovich SDEs within the SABR model  \citep{ref-yamada2019}. Crucially for our current task, Daveloose et al. \citep{ref-daveloose2019} derived an explicit expression for conditional expectations in Lévy and jump-diffusion contexts, enabling efficient pricing of American options and their sensitivities. Most recently, Kharrat \citep{ref-kharrat2024} has leveraged Malliavin calculus to evaluate conditional expectations for American put options in models featuring both stochastic volatility and stochastic interest rates, where both factors follow the Cox-Ingersoll-Ross dynamics.\\
\noindent This paper makes a significant advance by introducing a novel setting for computing conditional expectations in the context of mean-field stochastic differential equations  (SDEs) featuring jumps. Before this work, the application of Malliavin calculus in this domain has primarily relied on differentiation with respect to the Wiener process for systems without jumps. Subsequently,  in \citep{ref-daveloose2019}, the authors applied a similar approach to Lévy and jump-diffusion SDEs with a Wiener-Malliavin weight. Our methodology builds upon existing literature associated with Malliavin’s differential and integral calculus in Poisson space. Specifically, we employ directional differentiation with respect to the Poisson process—a technique whose systematic application in the mean-field setting is entirely new.\\
\noindent The article is organized as follows. 
In Section 1, we introduce mean-field SDEs with jumps, cite references confirming the existence and uniqueness of their solutions, and prove the existence of the associated stochastic flow, demonstrating an explicit representation for the directional Malliavin derivatives within this framework. Section 2 is devoted to the core methodology contribution: representing conditional expectations in terms of appropriately weighted ordinary expectations via our adapted Malliavin calculus, and examining the wider implications for computational performance in mean-field SDEs with jumps. Section 3 explores the integration of localization methods into our conditional expectation representations. leading to numerically efficient expressions in mean-field SDEs with jumps. Section 4 culminates with a thorough application of the derived representation formulas to the valuation of American put options governed by mean-field SDEs with jumps, followed by an analysis of the numerical results from the proposed computational procedure. \\
Detailed information on Malliavin calculus in Poisson space can be found in the exhaustive references \citep{ref-nualart2018}, \citep{ref-song2015}, \citep{ref-song2018}. We summarized various notations, definitions, computational rules, and properties, succinctly, in \ref{appendixA}. We will elucidate the practical implementation of the algorithm for pricing American options for our analysis in \ref{appendixC}.
\section{Mean-field SDEs with jumps}
\noindent In this section, we will introduce mean-field stochastic differential equations with jumps, providing a comprehensive framework for analyzing complex stochastic systems. We present the Malliavin differentiability of their solution and find its expression. \\
\noindent 
 Let $ \Theta \subset \mathbb{R}^d $ be an open set containing the origin, $\Theta_0 := \Theta \setminus \{0\}$. Consider the space \((\Omega, \mathcal{F}, \mathcal{P})\) as the Wiener-Poisson complete probability space, where \(\Omega = \Omega_1 \times \Theta_0\) and \(\mathcal{F} := \{\mathcal{F}_t \mid t \in [0, T]\}\) is the right-continuous and \(\mathcal{P}\)-complete natural filtration generated by independent \( d \)-dimensional Brownian motion and the Poisson random measure \(N(dt, dz)\) with intensity $ dt \nu(dz) $, where $\nu(dz) = \kappa(z) dz$ for $\kappa \in C^1(\Theta_0; (0, \infty)),$ satisfying 
 \begin{equation}\label{eqt2}
		|\nabla \log \kappa(z)| \leq \mathcal{C} \varrho(z) := \mathcal{C} (1 \vee \text{dist}(z, \Theta^c)^{-1}),  \quad \int_{\Theta_0} (1 \wedge |z|^2) \kappa(z) dz < \infty,
	\end{equation}
in which $C^1(\mathcal{X}; \mathcal{Y})$ denotes the set of Borel measurable function $\mathcal{X} \rightarrow \mathcal{Y}$ that is continuously differentiable, and
 $ \text{dist}(z, \Theta^c) $ represents the distance from $ z $ to the complement of $ \Theta_0 $.	
\noindent Define the compensated Poisson process ${\tilde{N}}(dt, dz) := N(dt, dz) - \nu(dz) dt,$ and consider the following mean-field SDE with jumps:
\begin{align}\label{eqt3}
&dX_t ^x =b(t, X_t ^x, \rho_{X_t ^x})dt + \sigma(t, X_t ^x, \pi_{X_t ^x}) dW_t + \int_{\mathbb{R}_0}^{}{\lambda(t, X_t ^x, z, \eta_{X_t ^x})}\tilde{N}(dt, dz), ~~X(0)= x \in \Theta_0,
\end{align} 
where the functions $b:\Omega \times [0, T] \times \mathbb{R}^d \times \mathbb{R}^d\to \mathbb{R}^d, \sigma: \Omega \times  [0, T]\times \mathbb{R}^d \times \mathbb{R}^d \to \mathbb{R}^{d\times d}$, and $\lambda: \Omega \times [0, T]\times \mathbb{R}^d \times \mathbb{R}_0  \times \mathbb{R}^d \to \mathbb{R}^{d}$ are $\mathcal{F}$-measuable, and
\begin{equation*}
\rho_{X_t ^x} := \mathbb{E}[\phi(X_t ^x)], \hspace{5mm} \pi_{X_t ^x} := \mathbb{E}[\psi(X_t ^x)], \hspace{5mm} \eta_{X_t ^x}:= \mathbb{E}[\xi(X_t ^x)],
\end{equation*}
in which $\phi, \psi, \xi: \mathbb{R}^d \to \mathbb{R}^d$ are also measurable functions. We assume the following conditions for these functions throughout the paper and denote $|y|$ for the Euclidean norm in the associated space of $ y$. \\
\begin{assumption}\label{Assum1}
The functions \( b, \sigma, \lambda, \phi, \psi, \xi \) satisfy the following conditions:
\begin{enumerate}
  \item The functions $ b, \sigma, \phi, \psi$,  and \( \xi \) are continuously differentiable and adhere to a Lipschitz condition with a bounded parameter \( {C} \).
  \item The coefficients \( b, \sigma, \) and \( \lambda \) satisfy the linear condition and 
\begin{align*}
 |\sigma(t, X_t^x, \pi_{ X_t^x}) - \sigma(s, Z_s^z, \pi_{ Z_s^z})|^2  \leq \mathcal{C}(|t - s|^2 + |X_t^x - Z_s^z|^2 + \mathbb{E}[|X_t^x - Z_s^z|^2] \vert z \vert).
  \end{align*}
  \item   For the function \( \zeta : \mathbb{R}_0 \to \mathbb{R} \) satisfying \( \int_{\mathbb{R}_0} \zeta^p(z) \, \nu(dz) < \infty \) for every \( p \geq 1 \), 
\begin{align*}
  	|\lambda(t, X_t^x, z, \eta_{ X_t^x}) - \lambda(s, Z_s^z, z, \eta_{ Z_s^z})|^2  \leq \zeta(z)(|t - s|^2 + |X_t^x - Z_s^z|^2 + \mathbb{E}[|X_t^x - Z_s^z|^2]).
  \end{align*}  
  \end{enumerate}
\end{assumption}

\noindent Kurtz and Xiong \citep{ref-kurtz1999} established that, under Assumption \ref{Assum1}, where $\nu=0$, there exists a pathwise unique solution \( X^x = \{X_t^x \mid t \in [0, T]\} \) for every initial condition \( x \in \mathbb{R}^k \) to  mean-field SDE \eqref{eqt3}, satisfying \( \mathbb{E}[X^2(t)] < \infty \) for all \( t \). The extension to cases involving jumps has been achieved using similar arguments in \citep{ref-kurtz1999}, and has been proved in \citep{ref-song2022}. Thanks to Theorem 2.11 in \citep{ref-kunita2004}, or Theorem 3.1 in \citep{ref-song2022}, for any \( p \geq 2 \), there exists a positive constant \( \mathcal{C}_p \) such that
	\begin{align}
		\mathbb{E} \left( \sup_{0 < x \leq t} |X(x)|^p \right) \leq \mathcal{C}_p.\label{xp}
	\end{align}
\noindent Denote  \( \mathcal{L}^p(\Omega) \) be the space of all measurable random variables with the finite norm $\|F\|_{\mathcal{L}^p}:= \left(\mathbb{E}|F|^p\right)^{1/p}$.
Define \( \partial_j = \frac{\partial}{\partial x_j} g \), for \( i = 1, 2, \ldots, n \) and any \( g = g(t, x_1, x_2, \ldots, x_n) \), and let $\sigma_n$ is the $n$-th row of $\sigma$. We suppose for each \( x \in \mathbb{R}^d \), the functions $b$, and $\sigma_n$ are differentiable, and set, for $j=1, 2$,
\begin{align*}
\|\partial_j b\|_1 &:= \sup_{X_t^x \in \mathbb{R}^d} |\partial_j b(t, X_t^x, \rho_{ X_t^x})|, \qquad \|\partial_j \sigma_n\|_1 := \sup_{X_t^x \in \mathbb{R}^d} |\partial_j \sigma_n(t, X_t^x, \pi _{X_t^x})|,  \\
 \|\partial_j \lambda\|_1& := \sup_{X_t^x \in \mathbb{R}^d, z \in \Theta_0} |\partial_j \lambda(t, X_t^x, z, \eta_{X_t^x})|,  \qquad A_t ^{x} := {\partial_1} b(t, X_{t} ^x, \rho_{ X_{t} ^x}), \\
 B_{t}^x &:= {\partial_1} \sigma(t, X_{t} ^x, \pi_{ X_{t} ^x}),\qquad \qquad \qquad 
M_{t,z}^x := {\partial_1} \lambda(t, X_{t}^x, z, \eta_{ X_{t} ^x}).
\end{align*}
\begin{assumption}\label{Assum2}
\begin{enumerate}
\item For j=1, 2, the functions $\partial_j b$ and $\partial_j \sigma_n$ are Lipschitz functions, and it holds that 
       $\|\partial_j b\|_1 + \|\partial_j \sigma_n\|_1 < \infty,$
      \item The function \( \lambda : \Omega \times [0, T] \times \mathbb{R}^d \times \mathbb{R}^d \to \mathbb{R}^{d \times d} \), satisfies 
     \begin{align*}
    	\inf_{x \in \mathbb{R}^d, z \in \Theta_0} \left|\det\left(I + \partial_1 \lambda(s, x, z, \eta_x)\right)\right| \geq 0, \quad and \quad
    	\|\partial_j \lambda\|_1  < \infty,  \quad j=1, 2, 3. {\hspace{3mm}} 
    \end{align*}
    and for all \( t, s \in [0, T] \), and stochastic processes \( X_t^x \) and \( Z_t^z \), 
\begin{align*}\label{A8}
|\partial_j \lambda(t, X_t^x, z, \eta_{X_t^x}) - \partial_j \lambda(t, Z_s^z, z, \eta_{Z_s^z})| \leq \zeta(z)  \mathcal{C} (|t - s| + |X_t^x - Z_s^z| + \mathbb{E}[|X_t^x - Z_s^z|]).
\end{align*}
\end{enumerate}
\end{assumption}
\noindent Let \( Y_t^x \) be the solution to the homogeneous stochastic differential equation:
\begin{equation}\label{eqt5}
dY_t^x = A_t^x Y_t^x \, dt +  B_{t}^x Y_t^x \, dW_t + \int_{\mathbb{R}_0} M_{t,z}^x Y_t^x \, \tilde{N}(dz, dt), \qquad Y_0=1.
\end{equation}
%
Under Assumption \ref{Assum1} and according to Lemma 3.2 in \citep{ref-song2022}, 
it follows that \( Y_t^x \) is almost surely invertible for all \( t \in [0, T] \), which is satisfying
\begin{align*}
(Y_{t} ^x)^{-1} = 1&- \int_{0}^{t} (Y_{s} ^x)^{-1} \tilde{A}_s \, ds -  \int_{0}^{t} (Y_{s} ^x)^{-1} B_{s} ^x \, dW_{s} - \int_{0}^{t} \int_{\mathbb{R}^d} (Y_{s}^x)^{-1} M_{s,z} ^x (I + M_{s,z} ^x)^{-1} \, \tilde{N}(dz, ds).
\end{align*}
where
\[
\tilde{A}_t := A_{t} ^x - (B_{t} ^x)^2 - \int_{\mathbb{R}^d} (M_{t,z} ^x)^2 (I + M_{t,z} ^x)^{-1} \nu(dz).
\]
In addition, for \( p \geq 1 \)
\begin{align}
	\mathbb{E}\Big(\sup_{0 \leq t \leq T}\vert (Y_t^x)^{-1} \vert^p\Big)+ \mathbb{E}\Big(\sup_{0 \leq t \leq T}\vert Y_t^x\vert^p \Big) < \infty.\label{yy-1}
\end{align}
%
%
Now, assume that for all \( x \in \mathbb{R}^d \), \( \lambda(t, x, z,\eta_{x}) \) belongs to \( C^1([0,T]\times \mathbb{R}^d \times \Theta_0 \times \mathbb{R}^d; \mathbb{R}^d) \), and 
\begin{equation}\label{lambdanu}
 \int_{\Theta_0} |\partial_z \lambda(t, x, z, \eta_{x})|^{2} \nu(dz) \leq \mathcal{C}(1 + |x|^{2}). 
\end{equation}
\noindent Having established sufficient conditions on the coefficients, we now confirm the Malliavin differentiability of the solution. By invoking Theorem 11.4.3 and Proposition 7.5.2 from \citep{ref-nualart2018}, we affirm for any $r \leq t$
\begin{align*}
	D_{r,z_0} X_t = \partial_z \lambda(r, X_r, z_0, \eta_{X_r}) & +  \int_r^t A_s^x D_{r, z_0}X_s \, ds  + \int_r^t B_s^x  D_{r,z_0}X_s \, dW_s\\
&+ \int_r^t \int_{\mathbb{R}_0} M_{s,z}^x D_{r,z_0}X_{s-} \, \tilde{N}(dz, ds), 
\end{align*}
\begin{align*}
D_r^w X_t = \sigma (r, X_r, \pi_{X_r}) + \int_r^t A_s^x D_r^w(X_s) \, ds  + \int_r^t B_s^x  D_r^w(X_s) \, dW_s+ \int_r^t \int_{\mathbb{R}_0} M_{t,z}^x D_r^w(X_{s-}) \, \tilde{N}(ds, dz).
\end{align*}
\noindent By noting that $D_r X_t$ fulfills equation \eqref{eqt5} with the initial point $\sigma (r, X_r^x, \pi_{X_r^x})$, and in accordance with Lemma 3.3 in \cite{ref-song2022}, we deduce for every $v \in  \mathcal{V}^\infty$, defined in Appendix A, 
\begin{align*}
D_{\bf{v}}X_t^x & :=\int_0^1 \int_{\Theta_0} \langle v(s, z), D_{s, z}X_t \rangle_{\mathbb{R}^d} {N}(dz, ds) \\
&= Y_{t} ^x\int_0^1 \int_{\Theta_0} (Y_{s} ^x)^{-1} \partial_{z}\lambda(s; X_{s} ^x; z; \eta_{ X_{s} ^x}) (I + M_{s,z} ^x)^{-1} v(s,z) N(dz, ds),
\end{align*}
and 
\begin{align*}
\qquad D_r^w X_{t} ^x = Y_{t} ^x(Y_{r} ^x)^{-1} \sigma(r; X_{r} ^x, \pi_{ X_{r} ^x}) \mathbf{1}_{r \leq t}.
\end{align*}
\begin{theorem}
 For \(p > 1\) and \(\vartheta = (h, v) \in \mathcal{H}^\infty \times \mathcal{V}^\infty\), the process $D_{\vartheta} X_t \in W^{1, \vartheta, p}$, defined in the Appendix A.
\end{theorem}
\begin{proof}
In Appendix A, we note that 
\begin{align*}
D_\vartheta X_t & =\int_0^1 \langle h(s), D_s^w X_t  \rangle_{\mathbb{R}^d} ds + D_{\bf{v}}X_t^x.
\end{align*}
From Definition 2.1, applying the Cauchy-Schwarz inequality, Yung inequality, and Proposition 3.1 in \cite{ref-song2022}, we deduce 
\begin{align*}
\Vert D_\vartheta X_t \Vert_{\varphi; 1, p}^p& \leq \Vert \int_0^1 \langle h(s), D_s^w X_t  \rangle_{\mathbb{R}^d} ds \Vert_{\mathcal{L}^p}^p+ \Vert D_{\bf{v}}X_t^x \Vert_{\mathcal{L}^p}^p\\
& \leq \mathbb{E}\Big(\Big|\int_0^1 \vert \langle h(s)\vert^2 ds \int_0^1 \vert D_s^w X_t  \vert^2 ds\Big|^p\Big) +\Vert D_{\bf{v}}X_t^x \Vert_{\mathcal{L}^p}^p\\
& \leq \Vert  h\Vert_{\mathcal{L}_p^2}^2 \mathbb{E}\Big(\sup_{0 \leq t, r \leq 1}\vert Y^x_t (Y_r^x)^{-1}\vert^{2(p+q)}\Big)^{\frac{1}{2q}}  
\mathbb{E}\Big(\int_0^1 \vert  \sigma(r; X_{r} ^x, \pi_{ X_{r} ^x}) \vert^{4p} dr\Big)^{\frac{1}{2q}} +\Vert D_{\bf{v}}X_t^x \Vert_{\mathcal{L}^p}^p \\
&< \infty.
\end{align*}
where $\frac1p+\frac1q=1$, and we used \eqref{yy-1}, \eqref{lambdanu}, \eqref{xp} in the last inequality.
\end{proof}

\section{Conditional expectation expression in mean-field SDEs}
\noindent In this section, we endeavor to determine conditional expectations with ordinary expectations weighted by appropriate factors through the application of Malliavin calculus. This method utilizes advanced stochastic calculus techniques to convert intricate conditional expectations into more tractable forms. Moreover, we will examine the broader implications of this approach for addressing complex problems and enhancing computational efficiency in mean-field SDEs with jumps.  
\subsection{Conditional expectation}
\noindent In this subsection, we derive a representation for the conditional expectation associated with Poisson Malliavin weights. This representation is essential for understanding the conditional expectation and its implications in various applications, including financial modeling and statistical analysis.\\
Let \( C'_b \) denote the space of continuously differentiable functions with bounded derivatives,
and \( \mathcal{H}(x) \) represent the Heaviside step function adjusted by a constant \( c \in \mathbb{R} \), defined as
    \begin{align*}
    \mathcal{H}(x) := \mathbf{1}_{\{x \geq 0\}} + c.
    \end{align*}
\begin{theorem} \label{th1}
 Assume \( f \in C'_b \), and \( F \) and \( G \) be solutions of mean-field stochastic differential equations with jumps, both belonging to the Malliavin space \( W^{1, \vartheta, 2} \). Consider \(u(.,.) \in \mathcal{V}^{\infty -}\) satisfying \( f(F) u(.,.)
 \in \mathcal{L}^2(\Omega \times [0,T] \times \Theta_0)\), 
and the following integrability and normalization condition  
\begin{equation}  \label{condition1}
    \mathbb{E}\left[ \int_0^T \int_{\Theta_0} \langle u(t,z), D_{t,z} G \rangle_{\mathbb{R}^d}\, {N}(dz, dt) \mid \sigma(F, G) \right] = 1.  
\end{equation}  

\noindent Then, for any \( \alpha \in \mathbb{R} \), the conditional expectation of \( f(F) \) given \( G = \alpha \) is expressed as  
\begin{align} \label{eq7} 
    \mathbb{E}[f(F) \mid G = \alpha] = \frac{  
    \mathbb{E} \Big[ \mathcal{H}(G - \alpha) \Big(f(F) \Pi_1  
    +  f'(F) \Pi_2 \Big)\Big]  
    }{  
    \mathbb{E} \left[ \mathcal{H}(G - \alpha) \delta(u) \right]  
    },  
\end{align}  
where
\begin{align}\label{weights}
	&\Pi_1=\delta(u)= \int_{0}^{T}\int_{\Theta_0}^{}\frac{div(\kappa u)(s,z)}{\kappa(z)}\tilde{N}(dz, ds),\qquad \Pi_2= \displaystyle \int_{0}^{T}\int_{\Theta_0}^{} \langle u(t,z), D_{t,z}F \rangle_{\mathbb{R}^d} {N}(dz, dt).\notag\\
\end{align}
\end{theorem}

\begin{proof}  
\noindent The proof is inspired by Theorem 3.1 in \citep{ref-daveloose2019}, in accordance with our approach of the Malliavin derivative in the context of Poisson processes.\\
Consider the definition of $ \Phi_{\epsilon, n}$, $ \Psi_{\epsilon, n}$, and $\mathcal{H}_{\epsilon}$ defined in the Appendix A. 
By applying the relation \eqref{condition1}, in combination with the chain rule mentioned in Appendix A, we arrive at the following expression.
\begin{align*}
\mathbb{E}\Big[f(F) \Phi_{\epsilon, n}(G - \alpha)\big]&= \mathbb{E}\Big[f(F) \Phi_{\epsilon, n}(G - \alpha) \mathbb{E}\bigg[\int_{0}^{T} \int_{\mathbb{R}_0}^{} \langle u(t,z), D_{t,z}G \rangle_{\mathbb{R}_0}  \, {N}(dz, dt) \bigg| \sigma(F, G) \bigg]\bigg]\notag\\
&= \mathbb{E}\bigg[f(F) \int_{0}^{T} \int_{\Theta_0}^{} \Phi_{\epsilon, n}(G - \alpha) \langle u(t,z), D_{t,z}G \rangle_{\mathbb{R}_0} \, {N}(dz, dt) \bigg]\notag\\
&={\mathbb{E}\Big[f(F){D_{u}\Psi_{\epsilon, n}(G-\alpha)}\Big]}.
\end{align*}
\noindent Moreover, since \(|\Psi_{\epsilon,n}| \leq \epsilon + \epsilon c\), and \(u \in \mathcal{V}^{\infty -}\), we proceed by taking the limit and employing the duality formula (Theorem \ref{IBF}) to obtain
\begin{align*}
\lim_{n \rightarrow \infty }\mathbb{E}\Big[{f(F)\Phi_{\epsilon, n}(G-\alpha)}\Big]&=
{\mathbb{E}\Big[{D_{u}(\mathcal{H}_{\epsilon}(G- \alpha)f(F))-\mathcal{H}_{\epsilon}(G- \alpha)D_{u}(f(F)) }\Big]}\notag\\
&=- {\mathbb{E}\Big[ {\mathcal{H}_{\epsilon}(G- \alpha)\Big(f(F)\delta(u)+D_{u}(f(F))} \Big)\Big]}. \notag
\end{align*}
\noindent Applying the foregoing result in the particular case where \( f \equiv 1 \), we then establish that 
\begin{align*}
&\mathbb{E}\big[{f(F)|G= \alpha}\big]= \lim_{\epsilon \to 0^+} \lim_{n \rightarrow \infty }\frac{\mathbb{E}\Big[{f(F)\Phi_{\epsilon, n}(G-\alpha)}\Big]}{\mathbb{E}\Big[{\Phi_{\epsilon, n}(G-\alpha)}\Big]}\notag\\
&=\frac{\mathbb{E}\Big[{-\frac{1}{\epsilon} \mathcal{H}_{\epsilon}(G-\alpha)\Big( f(F)\Pi_1 + f^{\prime}(F)\displaystyle \int_{0}^{T}{\int_{\mathbb{R}_0}^{} \langle u({t,z}), D_{t,z}F \rangle_{\mathbb{R}^d}{N}(dt, dz)}\Big)}\Big]}{\mathbb{E}\Big[ {- \frac{1}{\epsilon} \mathcal{H}_{\epsilon}(G -\alpha) \delta(u)} \Big]}.
\end{align*}
\noindent Noting that \(\left| \frac{1}{\epsilon} \mathcal{H}_{\epsilon}(G-\alpha) \right| \leq 1+c\), together with \(\mathbb{E}\left[|\delta(f(F)\cdot u)|\right] < \infty\), \(\delta(u) \in \mathcal{L}^2(\Omega)\), and the convergence \(\frac{1}{\epsilon}\mathcal{H}_{\epsilon}(x) \to \mathcal{H}(x)\) as \(\epsilon \to 0\), we complete the proof via the dominated convergence theorem. 
\end{proof}


\begin{theorem}\label{th2}
Let \( F \) and \( G \) be in \( W^{1, \vartheta, 2} \) and let \( f \in C'_b \) such that \( f(F) \in \mathcal{L}^2(\Omega) \). Consider a process \(u(.,.) \in \mathcal{V}^{\infty -}\), ensuring \( f(F) u(.,.)
 \in \mathcal{L}^2(\Omega \times [0,T] \times \Theta_0)\), satisfying \eqref{condition1} and, in addition,
\begin{equation}\label{condition2}
\mathbb{E}\left[\int_{0}^{T}\int_{\Theta_0}^{} \langle u(t,z), D_{t, z} F \rangle_{\mathbb{R}^d} \, {N}(dt, dz) \bigg| \sigma(F, G)\right] = 0.
\end{equation}
Then the following representation holds for \( \alpha \in \mathbb{R} \),
\begin{equation}\label{eqt9}
\mathbb{E}[f(F) | G = \alpha] = \frac{\mathbb{E}[ \mathcal{H}(G - \alpha) f(F)\delta(u)]}{\mathbb{E}[\mathcal{H}(G - \alpha) \delta(u)]}.
\end{equation}
\end{theorem}
\noindent
\begin{proof}
\noindent The approach adopted in this proof closely resembles the one utilized in the proof of Theorem 3.2 as presented in \citep{ref-daveloose2019}. For an in-depth exploration, please refer to \ref{appendixB}.
\end{proof}


\subsection{Mean-field SDEs with jump model}
\noindent In this subsection, we extend their methodology to the computation of conditional expectations within the framework of mean‑field stochastic differential equations (SDEs) with jumps in a Poisson space, employing Malliavin derivatives together with the dual formula established for Poisson space. \\
\noindent We will proceed to compute the Malliavin weights outlined in \eqref{weights} for mean-field SDEs with jumps, employing the Skorohod integral and duality formula.
\begin{proposition}\label{th3}
Let \( f \in C_b^\prime \); for \( 0 < s < t < T \), and \( \alpha \in \mathbb{R} \). In the setting described by the mean-field stochastic differential equation with jumps \eqref{eqt3}, we assume that for every $p \geq 2$
\begin{align}\label{eqt11}
&\mathbb{E}\Bigg[ \int_{0}^{T} \int_{\Theta_0} {(1 \wedge |z|^2)^p} \left( \frac{1+\vert \partial_z(log \partial_z \lambda) \vert ^{p}+\vert 1+M_{r,z}\vert^p\varrho(z)^p+\vert\partial_z M_{r,z}\vert^p}{\vert \partial_z \lambda(r, X_r^x, z, \eta_{X_r^x})\vert^{p}} \right)  \, \nu(dz) dr\Bigg] < \infty, 
\end{align}
\noindent Then the following representation holds for the conditional expectation:
\begin{equation*}
\mathbb{E} \left[ f(X_t ^x) \mid X_s ^x = \alpha \right] =
\frac{\mathbb{E} \left[ f(X_t ^x) \mathcal{H}(X_s ^x - \alpha) \Pi_1^\alpha  +  f'(X_t ^x) \mathcal{H}(X_s ^x - \alpha) \Pi_2^\alpha \right]}
{\mathbb{E} \left[ \mathcal{H}(X_s ^x - \alpha) \Pi_1 \right]},
\end{equation*}
\noindent where the Malliavin weights equal

\begin{align}\label{pi1}
\Pi_1^\alpha=  \frac{1}{Y_s^x}\int_0^s \int_{\Theta_0} {Y_r^x}\mathcal{J}(r,z)\tilde N(dz, \delta r),\quad
\Pi_2^\alpha =\frac{Y_{t}^{x}}{Y_{s}^{x}}\int_{0}^{s} \int_{\Theta_0}^{} (1 \wedge |z|^2) {N}(dz, dr),\hspace{70mm}
\end{align}
\noindent in which 
\begin{align*}
\mathcal{J}(r,z)& :=\frac{(1 \wedge |z|^2) }{\partial_z \lambda(r, X_r^x, z, \eta_{X_{r} ^x})}\Big[\Big(\frac{\partial_z \kappa(z) }{\kappa(z)} - \frac{\partial_z ^2 \lambda(r, X_r^x, z, \eta_{X_{r} ^x}) }{\partial_z \lambda(r, X_r^x, z, \eta_{X_{r} ^x})} \Big)+ {\frac{2|z|1_{z\leq 1}}{\partial_z \lambda(r, X_r^x, z, \eta_{X_{r} ^x})}}\Big](1+M_{r,z})\\
&+ \frac{(1 \wedge |z|^2) \partial_z M_{r,z}}{\partial_z \lambda(r, X_r^x, z, \eta_{X_{r} ^x})}.
\end{align*}
\end{proposition}

\begin{proof}
\noindent We will apply Theorem \ref{th1} for 
\begin{equation*}
	u(r,z) = \frac{Y_r ^x (I+M_{r,z}^x)(1 \wedge |z|^2) }{\mathfrak{A} Y_s ^x \partial_z \lambda(r, X_r ^x , z, \eta_{X_r ^x}) } \frac{1}{s} \mathbf{1}_{r\le s},
\end{equation*}
where $\mathfrak{A} =  \int_{\Theta_0} (1 \wedge |z|^2) \kappa(z) dz $. It is obvious that the condition \eqref{condition1} is fulfilled. To demonstrate that the process \( u(r,z) \)  is an element of the space \( \mathcal{V}^{\infty-} \),  we will establish two key criteria: the predictability of \( u(r,z) \) and the boundedness of its norm.\\
\noindent $\bullet$The components \( Y_r^x \), \( Y_s^x \), and \( X_r^x \) are adapted to the filtration \( \mathcal{F}_t \). The term \( \partial_z \lambda(r, X_r^x, z, \eta_{X_r}) \) is both deterministic and smooth, in accordance with Assumption \eqref{Assum1}. Furthermore, the indicator function \( \mathbf{1}_{\{r \leq s\}} \) possesses predictability. Consequently, \( u(r,z) \) qualifies as a predictable process.\\
\noindent $\bullet$
The expression of the gradient function associated with the function \(u\) is defined by
\begin{align*}
\nabla_z u = \frac{Y_r^x(I+M_{r,z}^x)}{\mathfrak{A}sY_s^x} \mathbf{1}_{\{r \leq s\}} \left[ \frac{2z1_{\vert z\vert \leq 1}}{\partial_z \lambda} - \frac{(1 \wedge |z|^2) \nabla_z (\partial_z \lambda)}{(\partial_z \lambda)^2} \right] + u(r,z) \frac{\partial_z M_{r,z}}{I+M_{r,z}^x} .  
\end{align*}  
\noindent In accordance with the assumption \eqref{eqt11}, it follows that $\mathbb{E}[ \int_{0}^{T} \int_{|z| \leq 1} \vert \frac{ z } {\partial_z \lambda(r, X_r^x, z, \eta_{X_{r} ^x})}\vert^p  \, \nu(dz) dr] < \infty$.
Therefore, from \eqref{yy-1}, the first part of Assumption \ref{Assum2}, Burkholder-Davis-Gundy, and then Hölder inequalities, we derive that there exists  some constant $\mathcal{C}$  such that 
\begin{align*}
\| \nabla_z u\|_{\mathcal{L}_p^2}^p &= \mathbb{E}\left[\int_0^1 \int_{\Theta_0} |\nabla_z u|^2 \nu(dz) dr\right]^{p/2} + \mathbb{E}\left[\int_0^1 \int_{\Theta_0} |\nabla_z u|^p \nu(dz) dr\right]\\
		&\leq \mathcal{C} \mathbb{E}\Big[ \sup_{r \leq s} \left| {Y_r^x}{(Y_s^x)^{-1}} \right|^p   \int_0^1 \int_{\Theta_0} \vert \frac{z} {\partial_z \lambda(r, X_r^x, z, \eta_{X_{r} ^x})}\vert^p  1_{|z| \leq 1} \, \nu(dz) dr  \Big] \\
		&\quad + \mathcal{C} \mathbb{E}\Big[ \sup_{r \leq s} \left| {Y_r^x}{(Y_s^x)^{-1}} \right|^p   \int_0^1 \int_{\Theta_0} \vert\frac{\partial_z (log \partial_z \lambda)} {\partial_z \lambda(r, X_r^x, z, \eta_{X_{r} ^x})}\vert^p ( 1 \wedge |z|^2)^p \, \nu(dz) dr  \Big] \\
&\quad + \mathcal{C} \mathbb{E}\Big[ \sup_{r \leq s} \left| {Y_r^x}{(Y_s^x)^{-1}} \right|^p   \int_0^1 \int_{\Theta_0} \vert \frac{(\partial_z M_{r,z})^p} {\partial_z \lambda(r, X_r^x, z, \eta_{X_{r} ^x})}\vert^p ( 1 \wedge |z|^2)^p\, \nu(dz) dr  \Big] < \infty,			
\end{align*} 
and
\begin{align*}
\begin{aligned}
\|u \varrho\|_{\mathcal{L}_p^2}^p &= \mathbb{E}\left[\int_0^1 \int_{\Theta_0} |u \varrho|^2 \nu(dz) dr\right]^{p/2}+\mathbb{E}\left[\int_0^1 \int_{\Theta_0} |u \varrho|^p \nu(dz) dr\right]\\
&\leq \mathcal{C} \mathbb{E}\Big[ \sup_{r \leq s} \left| {Y_r^x}{(Y_s^x)^{-1}} \right|^{2p} \Big]^\frac12  \mathbb{E}\Big[\int_0^1 \int_{\Theta_0} \vert\frac{1}{\partial_z \lambda(r, X_r^x, z, \eta_{X_{r} ^x})}\vert^{2p} (1 \wedge |z|^2)^{2p}\varrho(z)^{2p}\, \nu(dz) dr  \Big]^\frac12 \\
& + \mathcal{C} \mathbb{E}\Big[ \sup_{r \leq s} \left| {Y_r^x}{(Y_s^x)^{-1}} \right|^{2p} \Big]^\frac12  \mathbb{E}\Big[\int_0^1 \int_{\Theta_0} \vert\frac{(I+M_{r,z})}{\partial_z \lambda(r, X_r^x, z, \eta_{X_{r} ^x})}\vert^{2p} (1 \wedge |z|^2)^{2p}\varrho(z)^{2p} \, \nu(dz) dr  \Big]^\frac12 < \infty.\\
\end{aligned}
\end{align*}  
\noindent We conclude that both the integrals and the expectations are finite, leading us to the result \(\|u\|_{\mathcal{V}^p} < \infty\) for all \(p \geq 1\), and
\begin{align*}
u \in \bigcap_{p \geq 1} \mathcal{V}^p = \mathcal{V}^{\infty-}.
\end{align*}

\noindent Finally, we compute th initial weight in \eqref{pi1} corresponds to the Skorohod integral of $u$.  
\begin{align*}
\Pi_1= \delta(u) &= \int_{0}^{T} \int_0^t \frac{\partial_z (\kappa u(r, z))}{\kappa(z)} \, \tilde{N}(dz, \delta r) \\
&= \int_{0}^{t} \int_{\Theta_0} \frac{\partial_z \kappa(z) u(r, z) + \kappa(z) \, \partial_z u(r, z)}{\kappa(z)} \, \tilde{N}(dz, \delta r)\notag\\
&= \frac{1}{\mathfrak{A}sY_s^x}\int_0^s \int_{\Theta_0}  {Y_r^x}{\mathcal{J}(r,z)}\tilde N(dz, \delta r)= \frac{1}{\mathfrak{A}s}\Pi_1^\alpha.
\end{align*}
\noindent Thus, the relation given in \eqref{pi1} is established.
\noindent The second weighting term in \eqref{weights} is expressed as
\begin{align*}
\Pi_2&= \displaystyle \int_{0}^{T}{\int_{\Theta_0}^{}\langle u(r,z), D_{r,z}X_t ^x \rangle_{\mathbb{R}^d}{N}(dz, dr)} \notag\\
&= \int_{0}^{T}{\int_{\mathbb{R}_0}^{}{\frac{Y_r ^x (1 \wedge |z|^2)}{\mathfrak{A}sY_s ^x \partial_z \lambda(r, X_r ^x , z, \eta_{X_r ^x})}\mathbf{1}_{r \le s}}{Y_t ^x (Y_r ^x)^{-1} \partial_z \lambda(r, X_r ^x , z, \eta_{X_r ^x}) \mathbf{1}_{r \le t}}{ N}(dz, dr)}\notag\\
& = \frac{1}{\mathfrak{A}s} \int_{0}^{s}{\int_{\Theta_0}^{}{\frac{Y_t ^x (1 \wedge |z|^2)}{ Y_s ^x}}{N}(dz, dr)}=\frac{1}{\mathfrak{A}s}\Pi_2^\alpha.
\end{align*}
\end{proof}

\noindent Notably, Theorem \ref{th2} extends to this context, offering a practical tool for cases where $ f $ is non-differentiable.

\begin{proposition}\label{th5}
\noindent Let us revisit the model prescribed by the stochastic differential equation \eqref{eqt3}. For any Borel-measurable function $ f $ satisfying $ f(X_t) \in \mathcal{L}^2(\Omega) $, for $ 0 < s < t < T $, and $ \alpha \in \mathbb{R} $, the following holds under assumptions \eqref{eqt11},
\begin{align*}
\mathbb{E}\!\left[ f(X_t)\,|\,X_s = \alpha \right]
= \frac{\mathbb{E}\!\left[ f(X_t) \mathcal{H}(X_s - \alpha)\,\Pi \right]}{\mathbb{E}\!\left[ \mathcal{H}(X_s - \alpha)\,\Pi \right]},
\end{align*}
where the Malliavin weight $\Pi$, from that in \eqref{pi1}, defines as follows
\begin{align}\label{Pi}
\Pi= \frac{1}{s\mathfrak{A}}\Pi_1^\alpha - \Big[\frac{1}{(t-s)\mathfrak{A}Y_s^x}\int_s^t \int_{\Theta_0} Y_r^x\mathcal{J}(r,z)\tilde N(dz, \delta r)\Big].
\end{align}
\end{proposition}
\begin{proof}
Theorem \ref{th2} lets us do it for the process
\begin{align}\label{uhat}
\hat{u}(r,z)
&= \frac{Y_r^{x}\, (1 \wedge |z|^{2})(I+M_{r,z})}
        {Y_s^{x}\, \partial_{z}\lambda\!\left(r, X_r^{x}, z, \eta_{X_r^{x}}\right)\, \mathfrak{A}}
   \left[ \frac{1}{s}\,\mathbf{1}_{\{r \le s\}}
          - \frac{1}{t-s}\,\mathbf{1}_{\{s < r \le t\}} \right] \notag\\
&= {u}(r,z)
   - \frac{Y_r^{x}\, (1 \wedge |z|^{2})(I+M_{r,z})}
          {Y_s^{x}\, \partial_{z}\lambda\!\left(r, X_r^{x}, z, \eta_{X_r^{x}}\right)\, \mathfrak{A}}
     \cdot \frac{1}{t-s}\,\mathbf{1}_{\{s < r \le t\}} .
\end{align}
\noindent Upon comparison with the intermediate process invoked in the proof of Proposition \ref{th3}, we conclude that $\hat{u}$ resides in the domain of the divergence operator, $\mathcal{V}^{\infty-}$. It follows that $\Pi$ is given by the expression below,
\begin{align*}
\delta(\hat{u})&= \delta(u) - \frac{1}{(t-s)\mathfrak{A}Y_s^x}\int_s^t \int_{\mathbb{R}_0} Y_r^x\mathcal{J}(r,z)\tilde N(\delta r, dz).
\end{align*}
\noindent This concludes the proof.\\
\end{proof}

\noindent To enhance understanding of the relationships and Malliavin weights thus obtained, the following example is presented.\\
\begin{example}
\noindent
Consider the following mean-field stochastic differential equation with jumps, initialized at $X_0^x = x_0$
\begin{align*}
dX_s^x &= a\big(\mathbb{E}[X_s^x] + X_s^x\big)\,ds + b X_s^x\,dW_s + \int_{\mathbb{R}_0} (\lambda_0(r, \eta_{X_r^x})\, z+\lambda X_r^x)\tilde{N}(dz, dr),
\end{align*}
where $a, \lambda$ and $b$ are constants and $\lambda_0$ is a suitable function depending on $t$, $X_t^x$ and its distribution $\eta_{X_t^x}$, whose inverse is in $\mathcal{L}^p[0,1]$, for all $p \geq 2$.
\noindent Suppose, furthermore, that $\kappa(z) = 1$ on $\Theta=[-\frac12, \frac12]$. Then $\varrho(z)=1_{|z| <\frac14} \frac{1}{|z|}+ 1_{\frac14 \leq |z| \leq \frac12}\frac{1}{\frac12 -|z|}$, and for every $p \geq 2$, the condition \eqref{eqt11} equals to 
$$\int_0^1 \frac{ 1}{\lambda_0^p(r, \eta_{X_r^x})} dr\int_{-\frac12}^{\frac12} |z|^{2p} (1+ 1_{|z| <\frac14} \frac{1}{|z|}+ 1_{\frac14 \leq |z| \leq \frac12}\frac{1}{\frac12 -|z|})^p dz< \infty.$$
Then the expression for the corresponding Malliavin-type weight $\Pi$ takes the form
\begin{align*}
\Pi =\, & \frac{1}{s\mathfrak{A} Y_s^x}
        \int_0^s \int_{\mathbb{R}_0}
        \frac{2|z|Y_r^x}{\lambda_0(r, \eta_{X_r^x})}
        \tilde N(dz, dr) 
        - \frac{1}{(t-s) \mathfrak{A} Y_s^x}
        \int_s^t \int_{\mathbb{R}_0}
        \frac{2|z|Y_r^x}{\lambda_0(r, \eta_{X_r^x})}
        \tilde N(dz, dr),
\end{align*}
\noindent where the normalization constant $\mathfrak{A}$ is given by $\mathfrak{A} = \int_{[-\frac12, \frac12]} |z|^2\,\nu(dz)=\frac14$, and $Y_s^x$ is a process which may be computed according to relation~\eqref{eqt5} and equivalents to 
$$\exp\{(a-b^\frac12 -\lambda)s+W_s\} (1+\lambda)^{\int_0^s \int_{[-\frac12, \frac12]} N(dz,dr)}  .$$

\noindent This example concretely illustrates the explicit form of the Malliavin weight $\Pi$ for a class of mean-field SDEs with jumps.
\end{example}


\section{Variance reduction}
\noindent In the preceding sections, we explored representations in which the random variables used to estimate expected values may exhibit considerable variance. To effectively mitigate this variance and significantly enhance the accuracy of Monte Carlo simulations, the strategic application of variance‑reduction techniques is demonstrably paramount. In \citep{ref-moreni2004}, Moreni introduces a novel variance-reduction method that applies importance sampling in the Monte Carlo pricing of American options, effectively integrating the Longstaff-Schwartz algorithm to enhance computational efficiency. In Subsection \ref{5.1}, we examine the localization technique within the framework of our conditional expectation representation. Although this approach has previously been explored by Hansen \citep{ref-hansen2005} and Daveloose et al. \citep{ref-daveloose2019}, we tailor it here to suit the specific setting of our analysis.
\subsection{Localisation}\label{5.1}
\noindent We adapt the localization technique outlined in \citep{ref-daveloose2019} for the Conditional Malliavin method presented in Theorem \ref{th2}. For the proofs related to the subsequent proposition, please refer to \citep{ref-daveloose2019}.

\begin{proposition}\label{th6}
\noindent Assume the setting of Theorem \ref{th2}, then for any continuous function with bounded derivative $\psi : \mathbb{R} \to [0,\infty)$ satisfying $\int_{\mathbb{R}} \psi(t)dt = 1$ and for all $\alpha \in \mathbb{R}$, we have

\begin{align}
\mathbb{E}[f(F) \mid G = \alpha] &= \frac{\mathcal J^\psi_{F,G}[f](\alpha)}{\mathcal J^\psi_{F,G}[1](\alpha)}, 
\end{align}
where $\mathcal J^\psi_{F,G}[ \cdot ](\alpha)$ is given by
\begin{align}
\mathcal J^\psi_{F,G}[.](\alpha) &= \mathbb{E}\left[ .(F)\Big(\psi(G - \alpha) + \delta(u)[\mathcal H(G - \alpha) - \Psi(G - \alpha)]\Big) \right],
\end{align}

\noindent where $\Psi(x) = \int_{-\infty}^x \psi(t)dt.$
\end{proposition}

\noindent Having established the localized representations of the formulas for conditional expectation, a key question follows: what is the optimal choice for the localizing function $\psi$? To identify the optimal function, we begin with the assumption that the additional constant \( c \) in the function \( \mathcal{H} \) is set to zero, i.e., \(\mathcal{H}(x) = \mathbf{1}_{\{x \geq 0\}}\). This simplifying assumption streamlines the subsequent analysis and critically facilitates the characterization of the optimal localizing function. Let \( Z \) denote the factor equal to \( 1 \). We consequently arrive at the expression of the practical expectation form

\begin{align}
\mathcal{J}_{F,G}^{\psi}[\cdot](\alpha) &= \mathbb{E}\left[ \mathcal{F}\left( \psi(G - \alpha)Z + \Pi[\mathcal H(G - \alpha) - \Psi(G - \alpha)] \right) \right],
\end{align}
\noindent is approximated via Monte Carlo simulation. Specifically, let \( N \) represent the number of simulated values $F^{q}, G^{q}, \Pi^{q}$, and $Z^q$ of  $F, G, \Pi$, and $Z$, respectively; thus, we can derive the following estimation
\begin{align}
\mathcal{J}^{\psi}_{F,G}[\cdot](\alpha) & \approx \frac{1}{N} \sum_{q=1}^{N} \cdot \left( F^{q} \right) \left( \psi(G^{q}-\alpha) Z^{q} \right) + \Pi^{q} \left[ \mathcal H(G^{q}-\alpha) - \Psi(G^{q}-\alpha) \right].
\end{align}
\noindent To effectively reduce variance, the goal is to minimize the integrated mean squared error associated with the localization function $\psi$. Therefore, we must pursue the following optimization problem (This criterion was first proposed by \citep{ref-kohatsu2002})
\begin{align}\label{I}
\inf_{\psi \in \mathcal{L}^1} I(\psi),
\end{align}
\noindent where 
\begin{align}\label{L1}
\mathbf{L}^1 &= \{ \psi : \mathbb{R} \to [0, \infty) : \psi \in C^1(\mathbb{R}), \psi(+\infty) = 0, \int_{\mathbb{R}} \psi(t)dt = 1 \} 
\end{align}
\noindent and \( I \) equals the integrated variance up to a constant (in terms of \( \psi \))

\begin{align}\label{II}
I(\psi) &= \int_{\mathbb{R}} \mathbb{E}\left[ {\cdot}^2 (f)\big(\psi(G - \alpha)Z + \Pi [\mathcal H(G - \alpha) - \Psi(G - \alpha)]\big)^2  \right]dx. 
\end{align}

\noindent The determination of the optimal localization function $\psi$ is presented in the following proposition. The optimal localization function will differ between the numerator and denominator, as the associated optimization problems are fundamentally distinct. The proof presented by \citep{ref-bally2005} can be readily extended to the current context.

\begin{proposition} The infimum of the optimization problem \eqref{I} with \( I(\psi) \) given by \eqref{II} and \( \mathcal{H}(x) = \mathbf{1}_{\{x>0\}} \), is reached at \( \psi^* \), where \( \psi^* \) is the probability density of the Laplace distribution with parameter \( \lambda^* \), i.e., for all \( t \in \mathbb{R} \),
\begin{align*}
\psi^*(t) = \frac{\lambda^*}{2} e^{-\lambda^* |t|},
\end{align*}
\noindent where
\begin{align}\label{lambda}
\lambda^* = \left( \frac{\mathbb{E}[\cdot^2(F)\Pi^2]}{\mathbb{E}[\cdot^2(F)Z^2]} \right)^{\frac{1}{2}}.
\end{align}
\end{proposition}
\noindent The localizing function delineated in the preceding proposition is optimal in terms of minimizing variance; however, it does not demonstrate optimality in numerical experiments regarding computational efficiency. Consequently, Bouchard and Warin \citep{ref-bouchard2012} proposed the use of the exponential localizing function
\begin{align*}
\psi(x) = \lambda^* e^{-\lambda^* x} \mathbf{1}_{\{x \geq 0\}},
\end{align*}
\noindent where $\lambda^*$ is given by \eqref{lambda}. In future discussions, we will elucidate how applying this function significantly reduces computational effort.

\subsection{Practical Variance Reduction for Mean-Field Models with Jump Components}

\noindent  In this section, we implement the localization method for mean-field SDEs with jumps to obtain numerically amenable representations of the relevant quantities. Drawing on Proposition \ref{th5} and \ref{th6}, we arrive at the following conclusions. Unless stated otherwise, we set $Z=1$ for notational convenience. \\
\noindent For every Borel‑measurable $ f $ such that $ f(X_t) \in \mathcal{L}^2(\Omega) $, for $ 0 < s < t < T $, and $ \alpha \in \mathbb{R} $, assuming that $X_t$ satisfies equation \eqref{eqt3}, it follows that

\begin{align}\label{vareq}
\mathbb{E}\!\left[ f(X_t)\,|\,X_s = \alpha \right]
= \frac{\mathbb{E}\left[ f(X_t)\Big(\psi(X_s - \alpha) + \Pi[\mathcal H(X_s - \alpha) - \Psi(X_s - \alpha)]\Big) \right]}{\mathbb{E}\Big[\psi(X_s - \alpha) + \Pi[\mathcal H(X_s - \alpha) - \Psi(X_s - \alpha)] \Big]},
\end{align}
\noindent where $\Pi = \delta(\hat{u})$ and $\hat{u}$ are prescribed in \eqref{Pi} and \eqref{uhat}, respectively. In accordance with \citep{ref-bally2005}, we define the optimal localization function as follows
\begin{align} \label{si}
\psi^{\hat{u}}(t) = \frac{\lambda^{\hat{u}}}{2} e^{-\lambda^{\hat{u}} |t|},
\end{align} 
\begin{equation}\label{lambdaa}
\lambda^{\hat{u}} = \left(
\frac{\mathbb{E}\!\left[ f^2(X_t)\,\big(\Pi_1 - \mathcal{G}_{s,t}\big)^2 \right]}
     {\mathbb{E}\!\left[ f^2(X_t) \right]}
\right)^{\frac{1}{2}},
\end{equation}
\noindent for
\begin{equation*}
\mathcal{G}_{s,t}:= \frac{1}{(t-s)\mathfrak{A}}
  \int_s^t \int_{\mathbb{R}_0}
     \frac{Y_r^x}{Y_s^x}
     \mathcal{J}(r,z) \tilde N(\delta r, dz).
\end{equation*}

\noindent We next describe how to numerically calibrate the $\lambda^{\hat{u}}$ and compute the associated weights using Malliavin calculus and Monte Carlo simulation, then implement the approach for American options and report numerical findings.\\
 \noindent Because the application of our proposed approach and its numerical evaluation both encompass the pricing of American options on a basket of two assets, and present the respective computational outcomes, it is natural to extend and rigorously examine the pertinent formulas in the multidimensional context. It should be stressed that the formula is provided solely in its variance-reduced version, arising from the implementation of the localization function.

\begin{proposition}
\noindent
Assume the setting of Theorem~\ref{th2}. Let $f$ be a Borel-measurable function such that $f(X_t) \in \mathcal{L}^2(\Omega)$, for $0 < s < t < T$ and $\alpha \in \mathbb{R}$. Suppose $X_t = (X^1_t, X^2_t, \dotsc, X^d_t)$, where each $X^i_t$, with $i = 1, \dotsc, d$, satisfies equation~\eqref{eqt3}. Then the following holds
\begin{align*}
\mathbb{E}\!\left[ f(X_t)\,|\,X_s = \alpha \right]
= \frac{\mathbb{E}\left[ f(X_t)\prod_{i=1}^{d}\left(\psi^i(X^i_s - \alpha) + \Pi^i\left[\mathcal{H}(X^i_s - \alpha) - \Psi^i(X^i_s - \alpha)\right]\right) \right]}{\mathbb{E}\left[\prod_{i=1}^{d}\left(\psi^i(X^i_s - \alpha) + \Pi^i\left[\mathcal{H}(X^i_s - \alpha) - \Psi^i(X^i_s - \alpha)\right]\right)\right]},
\end{align*}
\noindent
where, for each $X^i_t$, the quantities $\Pi^i$, $\psi^i$, and $\Psi^i$ are computed according to the relations previously described, utilizing the Malliavin weights derived in the preceding section.
\end{proposition}

\noindent The ensuing proposition outlines the method for establishing a locally optimal function in a multidimensional context.

\begin{proposition}
\noindent Let $\mathbf{L}_1$ is defined according \eqref{L1} and $\mathbf{L}_d = \{ \psi : \mathbb{R}^d \to [0, +\infty) ; \ \psi(x) = \prod_{i=1}^{d} \psi_i(x_i), \ \text{where } \psi_i \in \mathbf{L}_1, \text{ for any } i \}$. Then 
\begin{align*}
\inf_{\psi \in\mathbf{L}_d} I_d(\psi) = I_d(\psi^*)
\end{align*}
\noindent where 
\begin{align*}
I_d(\psi) &= \int_{\mathbb{R}^d} \mathbb{E}\left[ {\cdot}^2 (f){\prod_{i=1}^{d}}\big(\psi^i(X_s ^i - \alpha)Z + \Pi^i [\mathcal {H}^i(X_s ^i - \alpha) - \Psi^i(X_s ^i - \alpha)]\big)^2  \right]dx. 
\end{align*}
\noindent where $\psi^*(x) = \prod_{i=1}^{d} \psi^*_i(x_i)$, with $\psi^*_j(\xi) = {\lambda^*_j}/2 e^{-\lambda^*_j |\xi|}$ being a Laplace probability density function on $\mathbb{R}$ and $\lambda^*_j = \lambda^*_j[f]$ satisfies the following system of nonlinear equations:
\begin{align*}
\lambda^{* 2}_j = \frac{\mathbb{E}\left(f^2(X_t) \Pi_ j ^2 \prod_{i \ne j} \left[\lambda^2_i +\Pi_i ^2 \right]\right)}{\mathbb{E}\left(f^2(X_t) \prod_{i \ne j} \left[\lambda^2_i + \Pi_i ^2 \right]\right)}\hspace{5mm}for\quad j = 1, \ldots, d.
\end{align*}
\end{proposition}
\begin{proof}
\noindent Reference \citep{ref-bally2005} contains the pertinent proof and establishes the formulation of the Malliavin weights corresponding to the symbol $\Theta_{s,t, i}$.
\end{proof}
\noindent In accordance with this theorem, the result presented for $\hat{u}$ can be extended to a multidimensional framework. For any asset $i$ whose price process follows a mean-field SDE with jumps, the corresponding parameter $\lambda_i^{\hat{u}}$ is characterized.

\section{Implementations and simulations}
\noindent In this section, we meticulously present the application of representation formulas in the context of pricing American put options within mean-field stochastic differential equations with jumps. The subsequent section will detail the numerical evidence stemming from the empirical application of the methodology outlined in Appendix C.
\subsection{Numerical Expriments}
\noindent The core of this section involves implementing the Malliavin Monte Carlo method for pricing American options driven by a mean-field SDE with jumps. We selectively adopt the Euler–Maruyama method as the primary computational tool, valuing its established efficacy and advantageous computational footprint.
\noindent To rigorously assess the convergence properties of the Euler-Maruyama method as applied to the mean-field SDE with jumps, we draw upon the comprehensive discussions presented in \citep{ref-sun2021}. For the accurate computation of the Malliavin weights, the simulation must concurrently encompass both the asset process $X$ and the auxiliary process $Y$.
\noindent To ensure convergence and validate the accuracy of the results, these simulations were conducted with a sample size of \(M = 1000\) and small step sizes of \(\varepsilon = 2^{-12}\) were utilized to ensure that the numerical solution aligns effectively with the dynamics of the mean-field SDE with jumps. Utilizing the localization function technique for variance reduction, we rigorously apply the steps of the MMC algorithm to various examples sourced from \citep{ref-sun2021}. In the foregoing, we have analysed the following examples under two separate cases:
\begin{itemize}
  \item For a single underlying asset with price \(x \in \mathbb{R}_+\), the payoff of a standard one-dimensional American put option is
  \[
    \Phi(x) = (K - x)^{+}.
  \]

  \item For a two-asset underlying with price vector \(x = (x_1, x_2) \in \mathbb{R}_+^2\), the payoff of an American put on the maximum of the two assets is
  \[
    \Phi(x) = \bigl(K - \max(x_1, x_2)\bigr)^{+}.
  \]
\end{itemize}
\noindent The evolution of each asset price is described by a mean-field stochastic differential equation with jumps. The finite difference method is employed as a benchmark against which the efficiency of our method is evaluated.\\
\begin{example}\label{E1}
\noindent  Consider the following mean-field SDE with jumps with \(X_0 ^x= x_0\):

\begin{align*}
dX_s ^x &= a(\mathbb{E}[X_s ^x] + X_s ^x)ds + bX_s ^x dW_s + \int_{\mathbb{R}_0} c z (\mathbb{E}[X_s ^x] + X_{s^-} ^x) \tilde{N}(dz, ds),
\end{align*}
where \(a,b,c\in\mathbb{R}\) are constants, \(W\) is a standard Brownian motion, and \(\tilde{N}\) is the compensated Poisson measure associated with the Lévy measure \(\nu(dz)=\kappa(z)\,dz\). We assume that the Lévy density is uniform on a bounded support $\kappa(z) = \mathbf{1}_{\{|z|< \tfrac{1}{2}\}}$, implying that the jump amplitudes are uniformly distributed within the bounded domain \(|z|<\frac{1}{2}\).
\noindent The support of \(\nu\) is \(\Theta_0 = \mathbb{R}_0 = \mathbb{R}\setminus\{0\}\),
and its complement is \(\Theta^c = \{0\}\).
Hence, for any \(z \in \Theta_0\),
\begin{align*}
\mathrm{dist}(z,\Theta^c) = |z|, \qquad
\varrho(z) = 1 \vee |z|^{-1} = \max(1,|z|^{-1}).
\end{align*}
\noindent The function \(\varrho(z)\) controls the allowed growth of gradients of functions near the singular boundary \(z=0\).
\begin{figure}[H]

             \subfloat{\includegraphics[width=1\textwidth, keepaspectratio]{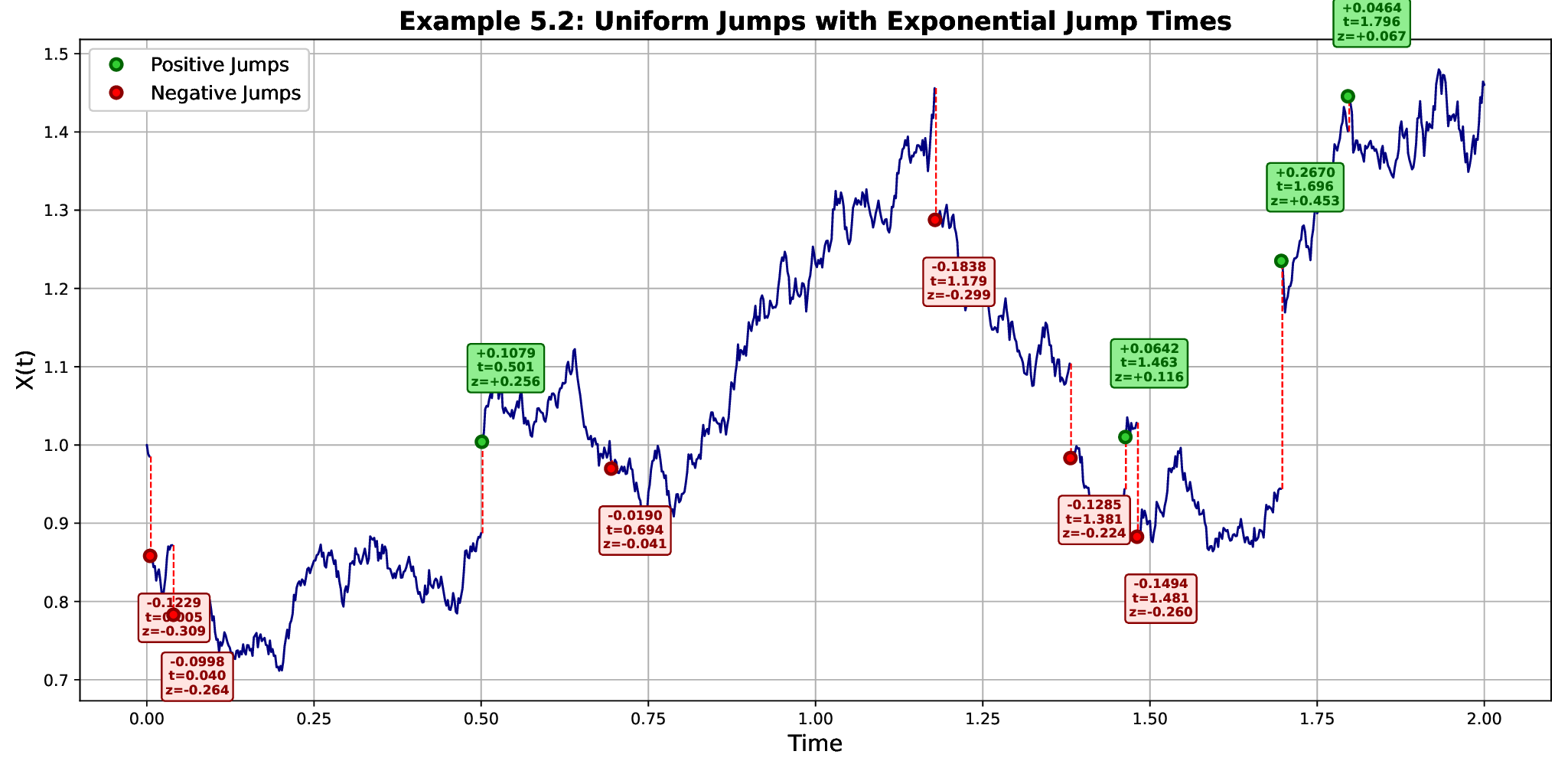}}
            \caption{\scriptsize On top parameters set to  $\lambda = 10, T = 1, N = 2^{12}$  Example\hspace{1mm}\ref{E1}.\\}
             {\scriptsize Jumps occur at times by an exponential distribution and jump sizes are in a uniform distribution $(-1/2, 1/2)$ }
 \label{EX1-Figure_1}
\end{figure}

\begin{figure}[H]

             \subfloat{\includegraphics[width=1\textwidth, keepaspectratio]{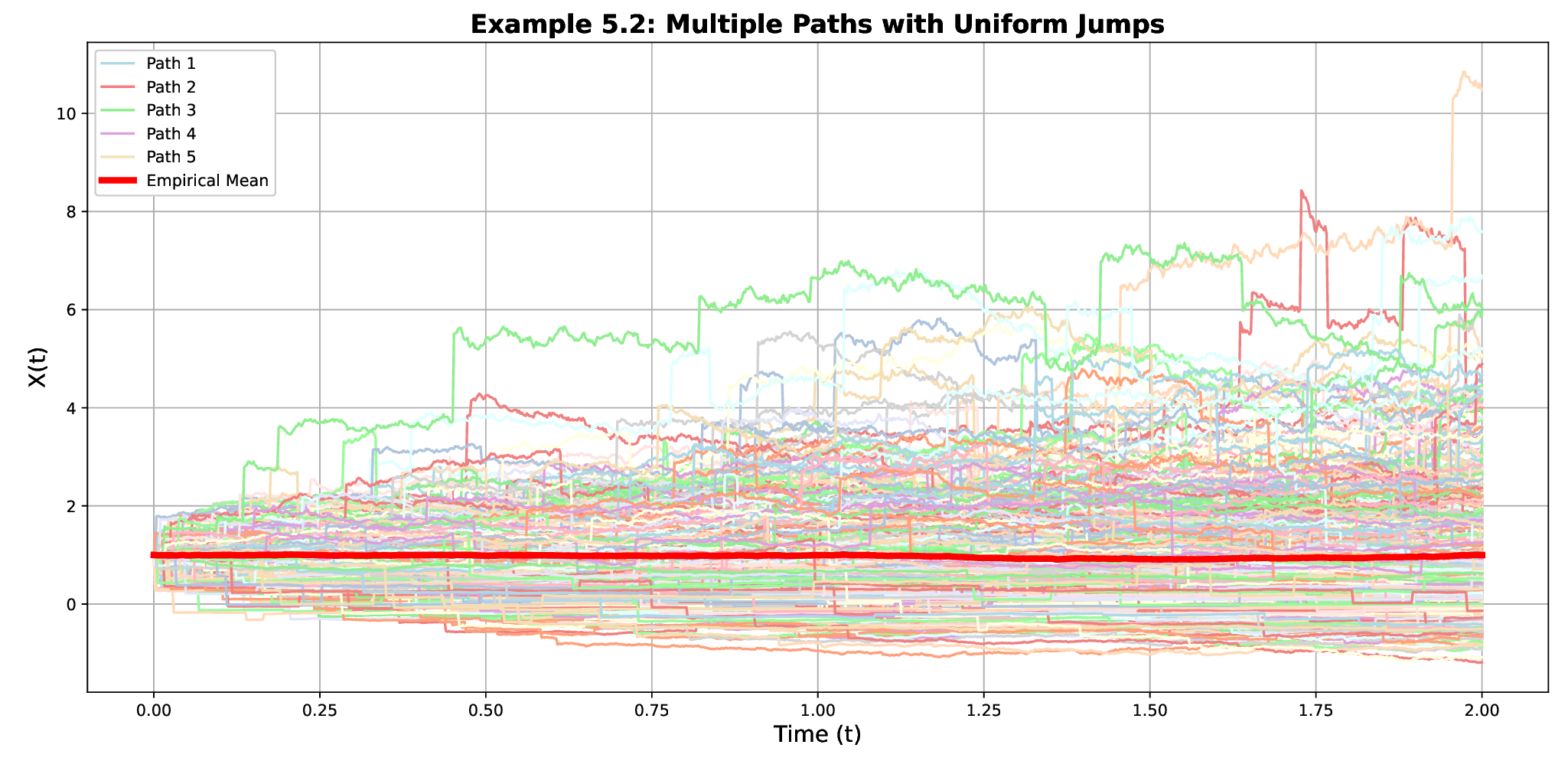}}
             \caption{\scriptsize On top parameters set to  $\lambda = 10, T = 1, N = 2^{12}$  Example\hspace{1mm}\ref{E1}.\\}
             {\scriptsize Jumps occur at times by an exponential distribution and jump sizes are in Uniform distribution $(-1/2, 1/2)$ }
 \label{EX1-Figure_2}
\end{figure}

\begin{figure}[H]

             \subfloat{\includegraphics[width=1\textwidth, keepaspectratio]{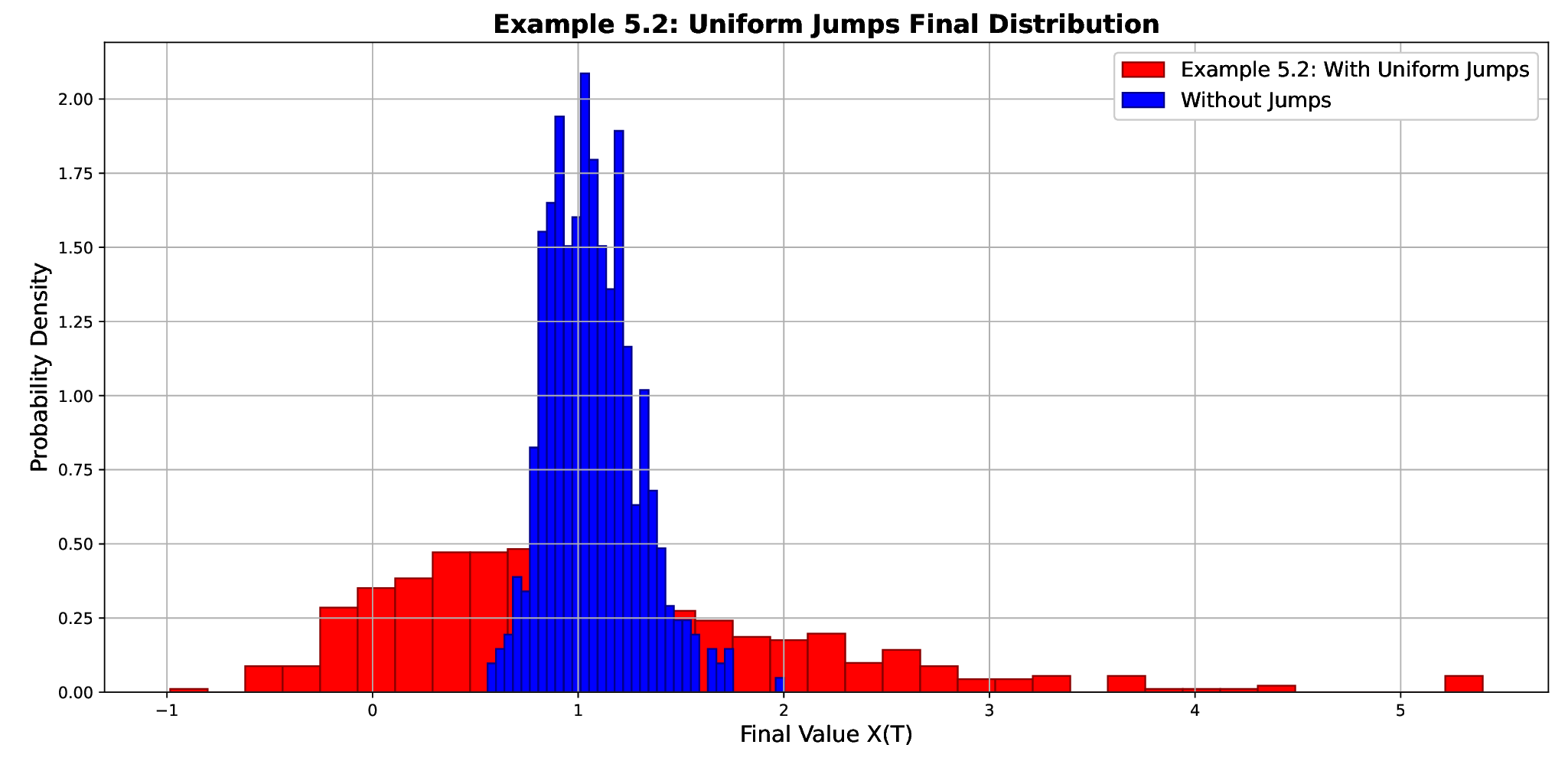}}
             \caption{\scriptsize On top parameters set to  $\lambda = 10, T = 1, N = 2^{12}$  Example\hspace{1mm}\ref{E1}.\\}
             {\scriptsize Jumps occur at times by an exponential distribution and jump sizes are in Uniform distribution $(-1/2, 1/2)$ }
 \label{EX1-Figure_4}
\end{figure}
\noindent It therefore follows in the present framework that $\varrho(z) = \max\{1, |z|^{-1}\}$. When $\hat{u}$ is specified in accordance with relation~\eqref{uhat}, one obtains
\begin{align*}
\hat{u}(r, z)
= (1+c)\frac{Y_r^x |z|^2}{Y_s^x \,\partial_z \lambda(z)\, \mathfrak{A}}
\left(
\frac{1}{s}\,\mathbf{1}_{\{r \leq s\}}
-
\frac{1}{t-s}\,\mathbf{1}_{\{s < r \leq t\}}
\right)
= C(s) {|z|^2}.
\end{align*}

\noindent where \(C(\cdot)\) is a bounded predictable function on \([0,1]\). We proceed to examine whether
\begin{align*}
\hat{u} \in \mathcal{V}^{\infty-}
\quad\Longleftrightarrow\quad
\|\nabla_z \hat{u}\|_{\mathcal{L}^2_p} < \infty 
\text{ and }
\|\hat{u}\varrho\|_{\mathcal{L}^2_p} < \infty,
\quad \forall\, p\ge 1.
\end{align*}
\noindent For the uniform law $\kappa(z)$, we have
\begin{align*}
\log \kappa(z) = 
\begin{cases}
0, & |z| < 1/2,\\[2pt]
-\infty, & |z|\ge 1/2,
\end{cases}
\qquad 
\nabla \log \kappa(z)=0 \text{ for } |z|<1/2.
\end{align*}
\noindent Therefore, the condition $|\nabla \log \kappa(z)|\le C\varrho(z)$ holds trivially inside the support.
Hence, the density $\kappa(z)$ is bounded and integrable.

\noindent The subsequent step entails the rigorous verification of the \(\mathcal{V}^p\)--norm conditions:
\paragraph{(i) The gradient term}
\noindent  Since $\nabla_z \hat{u}(s,z) = 2C(s)\,z,$ we have \(|\nabla_z \hat{u}(s,z)| = 2|C(s)||z|\).
Using the definition of the norm in \(\mathcal{L}^2_p\),
\begin{align*}
\|\nabla_z \hat{u}\|_{\mathcal{L}^2_p}^p
= \mathbb{E}\!\left[\left(\int_{\mathbb{R}_0} |\nabla_z \hat{u}(s,z)|^2
  \kappa(z)\,dz\right)^{p/2}\right].
\end{align*}
Since \(|C(s)|\) is bounded and \(\kappa(z)=\mathbf{1}_{\{|z|< 1/2\}}\),
\begin{align*}
\int_{\mathbb{R}_0} |\nabla_z \hat{u}(s,z)|^2 \kappa(z)\,dz
\le 4\,\|C\|_\infty^2 \!\int_{|z|< 1/2}\! z^2\,dz
= \frac{4\|C\|_\infty^2}{12}
<\infty.
\end{align*}
Hence, \(\|\nabla_z \hat{u}\|_{\mathcal{L}^2_p} < \infty\) for all \(p\ge 1\).

\paragraph{(ii) The weighted term}

\noindent Compute
\begin{align*}
|\hat{u}(s,z)|\varrho(z)
= |C(s)|\,|z|^2\,\max(1,|z|^{-1})
= |C(s)| \times
\begin{cases}
|z|, & |z| < 1,\\
|z|^2, & |z| \ge 1.
\end{cases}
\end{align*}
\noindent Since \(|z|\le 1/2\) on the support of \(\kappa\), only the first case applies
\begin{align*}
|\hat{u}(s,z)|\varrho(z) = |C(s)|\,|z|.
\end{align*}
\noindent Then
\begin{align*}
\|\hat{u}\varrho\|_{\mathcal{L}^2_p}^p
&= \mathbb{E}\!\left[\left(\int_{\mathbb{R}_0} 
   |\hat{u}(s,z)\varrho(z)|^2 \kappa(z)\,dz\right)^{p/2}\right] \\
&\le \mathbb{E}\!\left[\left(|C(s)|^2
   \int_{|z|\le 1/2} z^2\,dz\right)^{p/2}\right]
= \Bigl(\tfrac{|C(s)|^2}{12}\Bigr)^{p/2}
< \infty.
\end{align*}
\noindent Thus, \(\|\hat{u}\varrho\|_{\mathcal{L}^2_p} < \infty\) for all \(p\ge1\). Both parts of the norm are finite for each \(p\ge1\); therefore, $\hat{u}\in \mathcal{V}^{\infty-}$.

\begin{figure}[H]

             \subfloat{\includegraphics[width=1\textwidth, keepaspectratio]{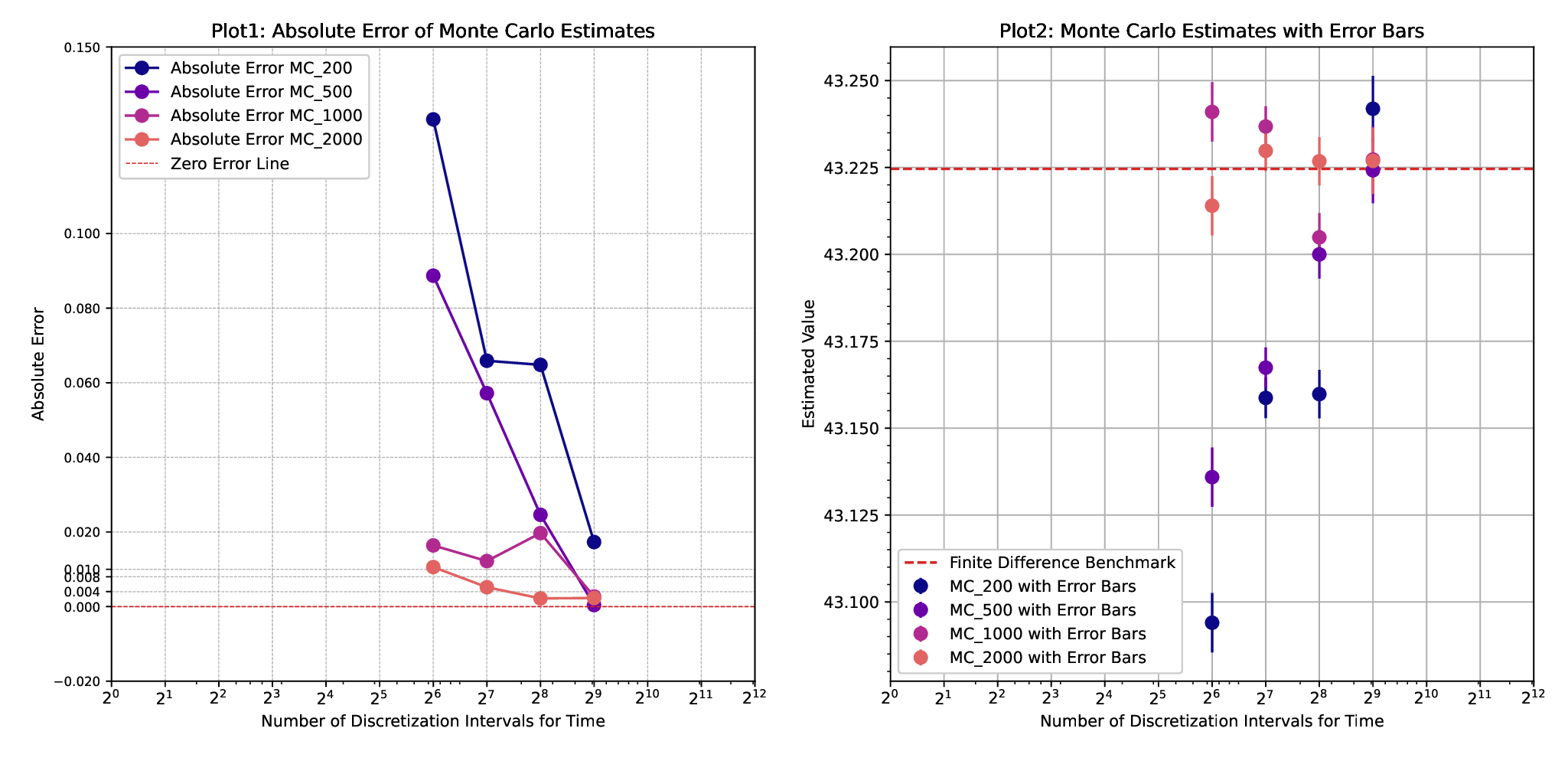}}
            \caption{\scriptsize Overview of Errors in Example\hspace{1mm}\ref{E1}.\\}
             {\scriptsize Analysis of MMC methods: (Plot1) Absolute error showing reduction with increased discretization and the number of Monte Carlo simulations; (Plot2) MMC estimates with error bars, underscoring variability and the reliability of the methods used. }
 \label{EX1-Error}
\end{figure}
\noindent The plots provide an in-depth analysis of the Malliavin Monte Carlo (MMC) method for pricing American options. In Figure {\rm \ref{EX1-Error}}, Plot 1 illustrates the absolute error between MMC estimates and the finite difference benchmark, revealing a pronounced decrease in error as both the number of discretization intervals and the number of simulations increase. This trend underscores the importance of the Law of Large Numbers, highlighting the need for sufficient sample sizes for enhanced estimation accuracy. Plot 2 displays the MMC estimates with error bars illustrating their variability. These error bars reflect how both the choice of sampling size and time discretization influence the stability of the estimates.

\begin{table}[H]
\caption{Evaluations of American Option Pricing in Context of Example\ref{E1}.}
	\begin{adjustbox}{max width=\textwidth}
		\begin{tabular}{ccccc|c}
			\toprule
			\textbf{Malliavin Monte-Carlo}	& \textbf{MC = 200}	& \textbf{MC = 500}     & \textbf{MC = 1000}     & \textbf{MC = 2000}     & \textbf{Finite Difference }\\
			\midrule
\multirow[m]{1}{*}{N = 64}	& 43.094			& 43.1359			& 43.2510			& 43.2040   &43.2246\\
			  	                   
                   \midrule
\multirow[m]{1}{*}{N = 128}    & 43.1587			& 43.1674			& 43.2568			& 43.2398   &43.2246\\

		\midrule	
\multirow[m]{1}{*}{N = 256}    & 43.1598			& 43.00			& 43.4049			& 43.2268   &43.2246\\

		\midrule	  	                 
\multirow[m]{1}{*}{N = 512}    & 43.2719			& 43.2242			& 43.2073			& 43.2289    &43.2246\\

			\bottomrule
		\end{tabular}
	\end{adjustbox}
	\noindent{\scriptsize{Data:$X^x = 1, K = 40$.}}
\end{table}
\end{example}


\begin{example}\label{E2}
\noindent  Let us examine the following mean-field stochastic differential equations with jumps, where the underlying process is initialized at \(X_0 ^x= x_0\):
\begin{align*}
dX_s ^x &= a\mathbb{E}[X_s ^x]ds + bX_s ^x dW_s + \int_{\mathbb{R}_0} |z|^2 \tilde{N}(dz, ds),
\end{align*}
\noindent where \(a\) and  \(b\) are constants. We set \(a = 1\), \(b = 0.5\), and \(x_0 = 1\). In this example, we focus on the Kou jump model. Specifically, we consider a L\'evy measure $\nu(dz) = \kappa(z)\,dz$, where
\begin{align*}
\kappa(z) = \lambda \big( p \eta_1 e^{-\eta_1 z} \mathbf{1}_{\{z>0\}} + (1-p)\eta_2 e^{\eta_2 |z|} \mathbf{1}_{\{z<0\}} \big),
\end{align*}
\noindent is the density of the Kou process, with parameters $\lambda > 0$, $p \in (0,1)$, and $\eta_1, \eta_2 > 0$.
\begin{figure}[H]

             \subfloat{\includegraphics[width=1\textwidth, keepaspectratio]{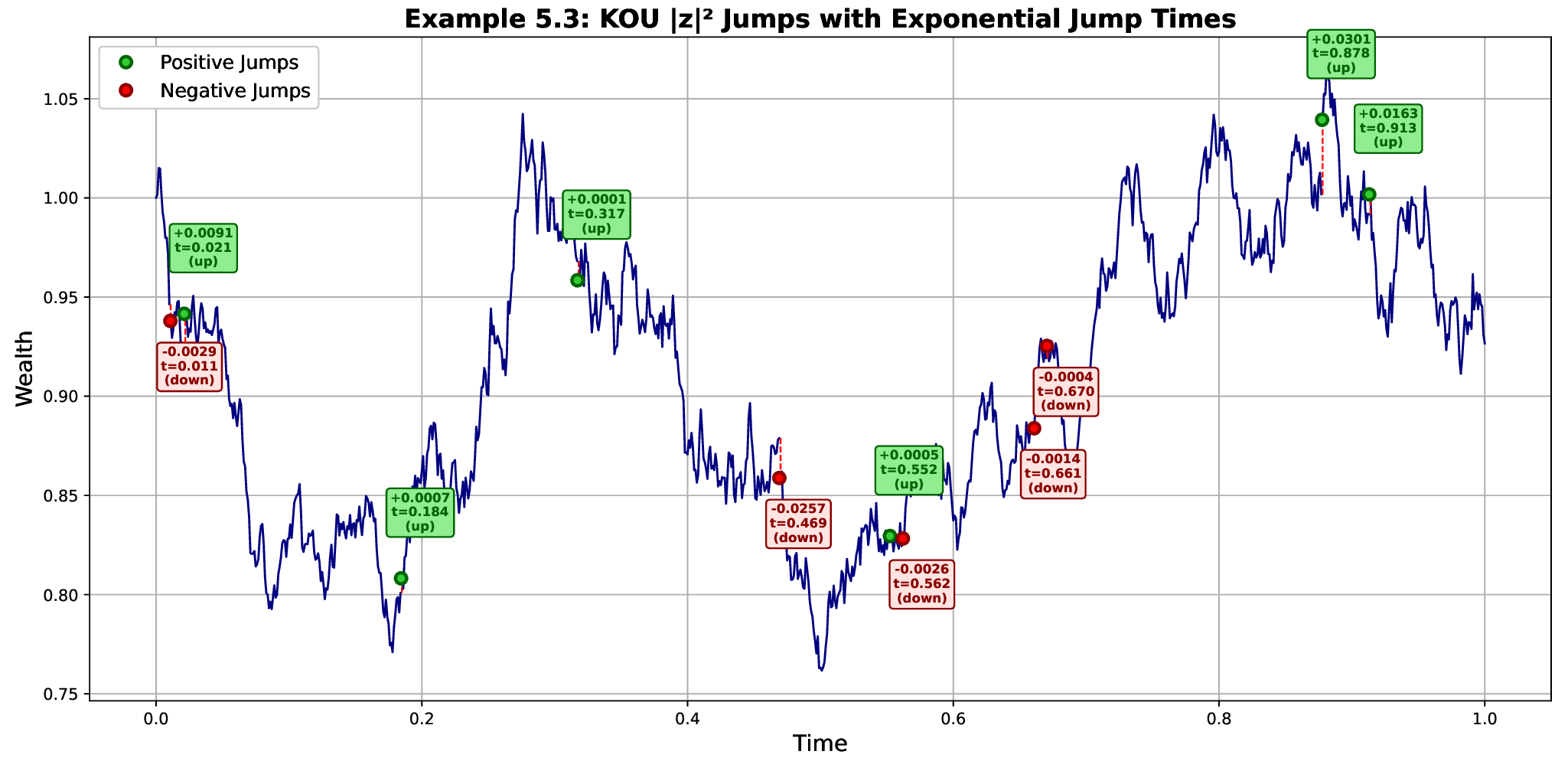}}

           \hspace{20mm}  \caption{\scriptsize On top parameters set to  $\lambda = 10, T = 1, N = 2^{12}$  Example\hspace{1mm}\ref{E2}.\\}
             {\scriptsize Jumps occur at times by an exponential distribution, and jump sizes are in the Kou model }
 \label{EX2-Figure_1}
\end{figure}
\unskip
\begin{figure}[H]

             \subfloat{\includegraphics[width=1\textwidth, keepaspectratio]{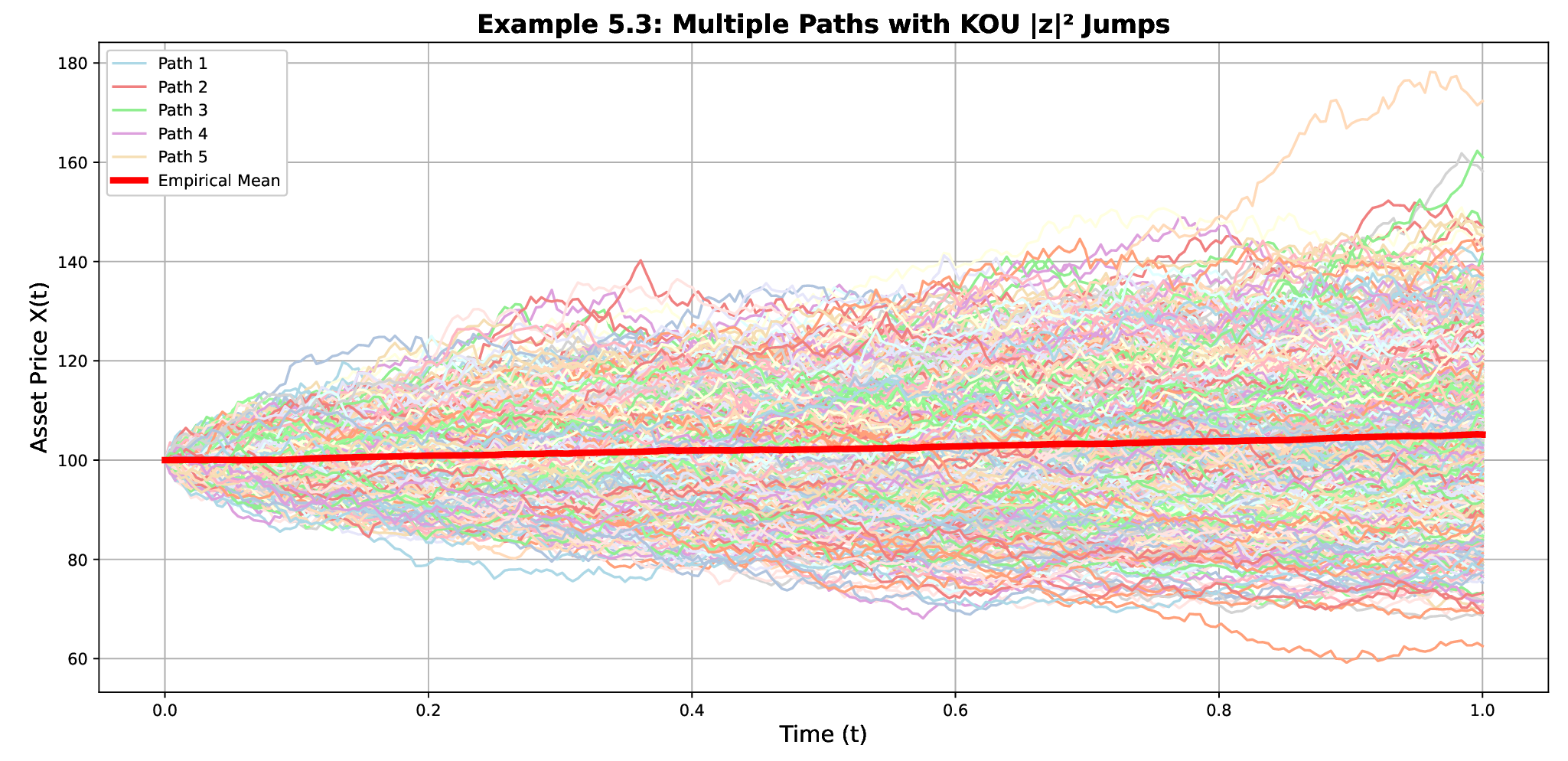}}
             \caption{\scriptsize On top parameters set to  $\lambda = 10, T = 1, N = 2^{12}$  Example\hspace{1mm}\ref{E2}.\\}
             {\scriptsize Jumps occur at times by an exponential distribution, and jump sizes are in the Kou model }
 \label{EX2-Figure_2}
\end{figure}
\begin{figure}[H]

             \subfloat{\includegraphics[width=1\textwidth, keepaspectratio]{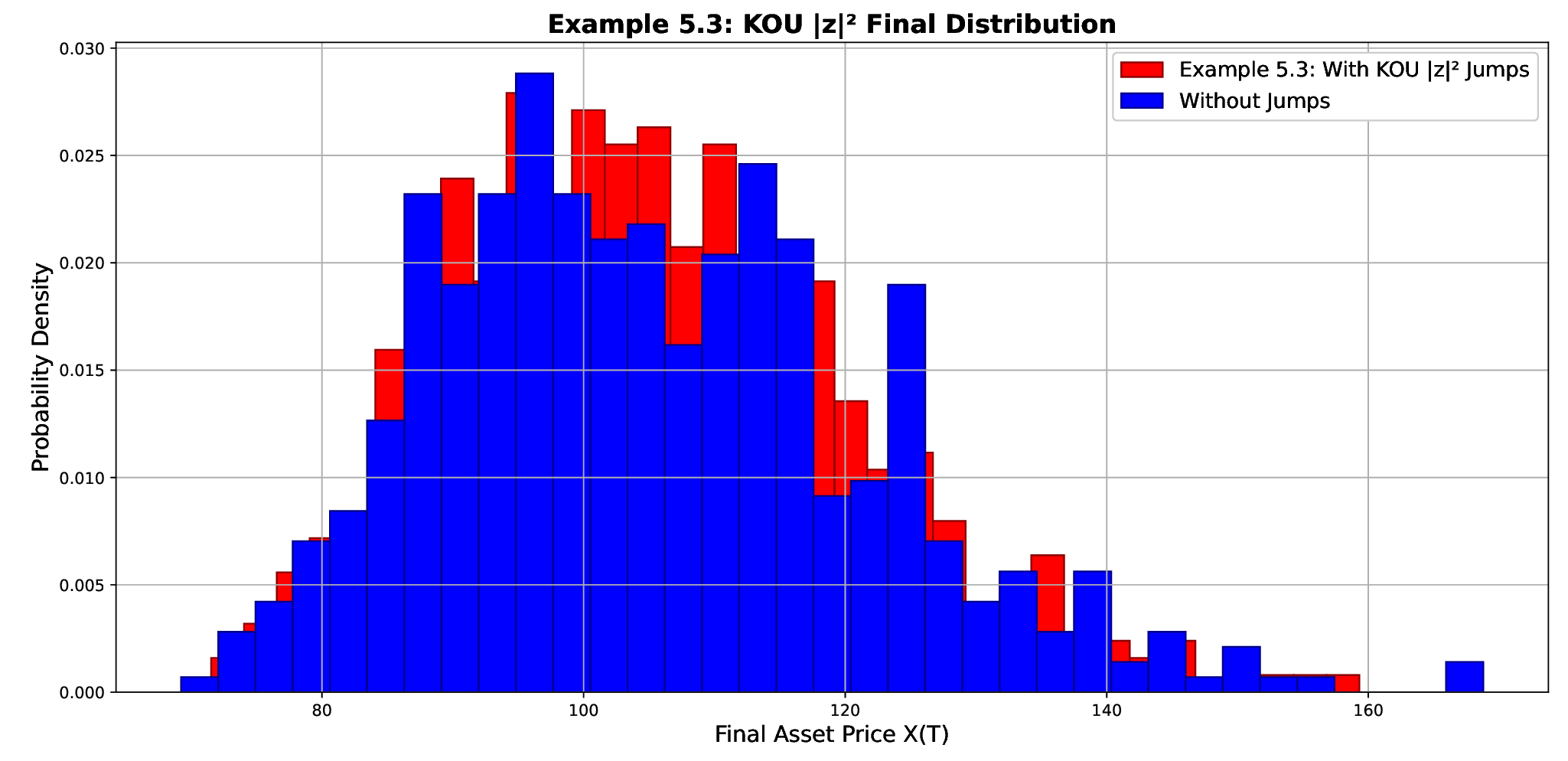}}

             \caption{\scriptsize On top parameters set to  $\lambda = 10, T = 1, N = 2^{12}$  Example\hspace{1mm}\ref{E2}.\\}
             {\scriptsize Jumps occur at times by an exponential distribution, and jump sizes are in the Kou model}
 \label{EX2-Figure_4}
\end{figure}
\noindent For the moment, we aim to verify that, in this setting, the function $\hat{u}$ satisfies the structural conditions required by the problem. Accordingly, we aim to establish that $\hat{u} \in \mathcal{V}^{\infty-}$ in the case of the Kou model, where we employ the definition
\begin{align*}
\varrho(z) = 1 \vee \operatorname{dist}(z, \Theta^c)^{-1}, \quad \Theta_0 = \mathbb{R}_0 \implies \operatorname{dist}(z, \Theta^c) = |z|.
\end{align*}
Hence, in this setting we have $\varrho(z) = \max\{1, |z|^{-1}\}.$ By selecting $\hat{u}$ in accordance with relation \eqref{uhat}, we observe that
\begin{align*}
\hat{u}(s, z)
= \frac{Y_r^x (1 \wedge |z|^2)}{Y_s^x \,\partial_z \lambda\, \mathfrak{A}}
\left(
\frac{1}{s}\,\mathbf{1}_{\{r \leq s\}}
-
\frac{1}{t-s}\,\mathbf{1}_{\{s < r \leq t\}}
\right)
= C(s)\,\frac{ (1 \wedge |z|^2)}{\partial_z \lambda(z)}.
\end{align*}
\noindent Suppose that $|\partial_z \lambda| \sim |z| \quad \text{as } z \to 0 \quad \text{(for instance, when } \lambda(z) = |z|^2 \text{)},$ which implies that
\begin{align*}
\hat{u}(s, z) \approx C(s)\,|z|
\quad \text{for } z \text{ in a neighbourhood of } 0.
\end{align*}
\noindent In the sequel, we shall establish the validity of $\hat{u} \in \mathcal{V}^{\infty-}:$
\begin{itemize}
\item[(i)] \text{Gradient term \(\|\nabla_z \hat{u}\|_{\mathcal{L}_{p} ^2}\): We know}
  \begin{align*}
    \nabla_z \hat{u} = C(s) \cdot sgn(z), \quad |\nabla_z \hat{u}| \leq K.
  \end{align*}
  \noindent Since \(\int_{\mathbb{R}_0} \nu(dz) < \infty\) (Kou's \(\kappa(z)\) is integrable),
  \begin{align*}
    \|{\nabla}_z \hat{u}\|_{\mathcal{L}_{p} ^2}
    &\leq \mathcal{C} \left( \mathbb{E}\left[ \left( \int_{\mathbb{R}_0} \nu(dz) \right)^{p/2} \right] \right)^{1/p}
       + \mathcal{ C}\left( \mathbb{E}\left[ \int_{\mathbb{R}_0} \nu(dz) \right] \right)^{1/p}
    < \infty.
  \end{align*}
\item[(ii)] \text{Weighted norm \(\|\hat{u} \varrho\|_{\mathcal{L}_{p} ^2}\):}
Compute \(|\hat{u} \varrho|\) in the form
  \[
    |\hat{u}(s,z)| \,\varrho(z)
    = |C(s)|\,|z| \cdot \max(1, |z|^{-1})
    = |C(s)| \cdot
    \begin{cases}
      |z|, & \text{if } |z| \geq 1, \\
      1,   & \text{if } |z| < 1.
    \end{cases}
  \]
  Now, evaluate \(\|\hat{u} \varrho\|_{\mathcal{L}^{p}}\) as follows.
  \[
    \int_{\mathbb{R}_0} |\hat{u} \varrho|^p \,\nu(dz)
    \leq K^p \left(
      \int_{|z| \geq 1} |z|^p \kappa(z)\,dz
      + \int_{|z| < 1} \kappa(z)\,dz
    \right).
  \]

For \(|z| \geq 1\) we know
    \begin{align*}
      &\int_{|z| \geq 1} |z|^p \kappa(z)\,dz < \infty \qquad \text{(exponential decay dominates polynomial growth),}
    \end{align*}

\text{and for \(|z| < 1\)}
    \[
      \int_{|z| < 1} \kappa(z)\,dz < \infty
      \quad \text{(Kou's \(\kappa(z)\) is bounded near \(0\)).}
    \]
\end{itemize}
\noindent Thus, \(\|\hat{u} \varrho\|_{L^{p}} < \infty\) for all \(p \geq 1\).
\begin{figure}[H]

             \subfloat{\includegraphics[width=1\textwidth, keepaspectratio]{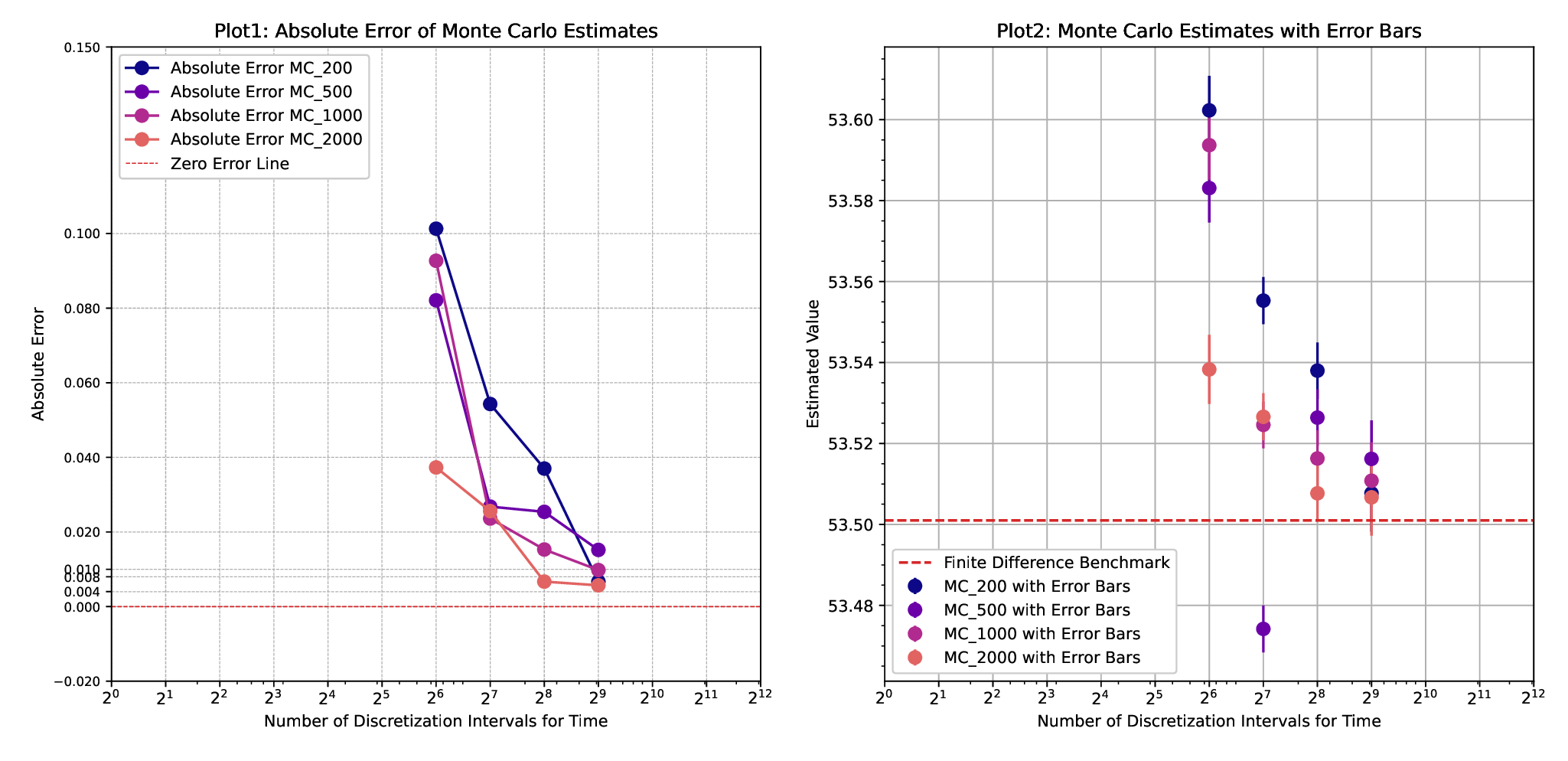}}
            \caption{\scriptsize Overview of Errors in Example\hspace{1mm}\ref{E2}.\\}
             {\scriptsize Analysis of MMC methods: (Plot1) Absolute error showing reduction with increased discretization and the number of Monte Carlo simulations; (Plot2) MMC estimates with error bars, underscoring variability and the reliability of the methods used. }
 \label{EX2-Error}
\end{figure}

\noindent The plots provide an in-depth analysis of the Malliavin Monte Carlo (MMC) method for pricing American options. In Figure {\rm \ref{EX1-Error}}, Plot 1 illustrates the absolute error between MMC estimates and the finite difference benchmark, and Plot 2 displays the MMC estimates with error bars illustrating their variability. 

\begin{table}[H]
\caption{Evaluations of American Option Pricing in Context of Example\ref{E2}.}
	  \begin{adjustbox}{max width=\textwidth}
		\begin{tabular}{ccccc|c}
			\toprule
			\textbf{Malliavin Monte-Carlo}	& \textbf{MC = 200}	& \textbf{MC = 500}     & \textbf{MC = 1000}     & \textbf{MC = 2000}     & \textbf{Finite Difference }\\
			\midrule
\multirow[m]{1}{*}{N = 64}	& 53.6023			& 53.5831			& 53.5937			& 53.5383			& 53.5010\\
			  	                   
                   \midrule
\multirow[m]{1}{*}{N = 128}    & 53.55533		& 53.4742			& 53.5246			& 53.5266     & 53.5010\\

		\midrule	
\multirow[m]{1}{*}{N = 256}    & 53.5380			& 53.5264			& 53.5163			& 53.5077      & 53.5010\\

		\midrule	  	                 
\multirow[m]{1}{*}{N = 512}    & 53.5077			& 53.5162			& 53.5108		& 53.5067       & 53.5010\\

			\bottomrule
		\end{tabular}
	\end{adjustbox}
	\noindent{\scriptsize{Data: $x = 10, K = 60$.}}
\end{table}

\end{example}

\begin{example}\label{EX3}
\textbf{Portfolio optimization in mean-field SDEs with jumps}\\
\noindent  Consider a two-asset portfolio optimization problem where asset prices follow mean-field jump-diffusion processes. Let $X_1^x$ and $X_2^y$ represent the prices of two risky assets with initial values $X_0^x = x_0 = 1$ and $X_0^y = y_0 = 10$:

\begin{align*}
&dX_s^x = \mu_1(X_s^x, \mathbb{E}[X_s^x])ds + \sigma_1 X_s^x dW_s^1 + \int_{\mathbb{R}} \lambda_1(t, X_t ^x, z, \mathbb{E}[X_t ^x]) \tilde{N}_1(dz, ds),  \\
&dX_s^y = \mu_2(X_s^y, \mathbb{E}[X_s^y])ds + \sigma_2 X_s^y dW_s^2 + \int_{\mathbb{R}} \lambda_2(t, X_t ^y, z, \mathbb{E}[X_t ^x]) \tilde{N}_2(dz, ds),  
\end{align*}

\noindent where
\begin{itemize}
    \item $\mu_i(x, m) = a_i x + b_i m + c_i$ denotes the mean-field drift function, where $a_i, b_i, c_i \in \mathbb{R}$ are structural parameters governing the linear dependence on current asset price and collective market behavior
    \item $\sigma_i > 0$ represents the diffusion coefficients capturing the continuous volatility dynamics of each asset
    \item $W^1$ and $W^2$ constitute independent standard Brownian motions driving the stochastic fluctuations in asset prices
    \item $\tilde{N}_i(dz, ds)$ are compensated Poisson random measures with Lévy densities $\nu_i(dz) = \kappa(z)dz$, where $\kappa(z)$ follows a uniform distribution for Asset 1 and a Kou double-exponential distribution for Asset 2, respectively
    \item $\lambda_1(t, X_t ^x, z, \mathbb{E}[X_t ^x]) = c z (\mathbb{E}[X_s^x] + X_{s^-}^x)$ and $\lambda_2(t, X_t ^y, z, \mathbb{E}[X_t ^y]) = |z|^2$ specify the state-dependent jump intensities, where $\lambda_1$ exhibits linear dependence on both the mean-field and pre-jump asset values, while $\lambda_2$ demonstrates quadratic scaling with jump magnitude
\end{itemize}

\begin{figure}[h]

             \centering
\begin{adjustbox}{width=1.3\textwidth,center}
            \includegraphics{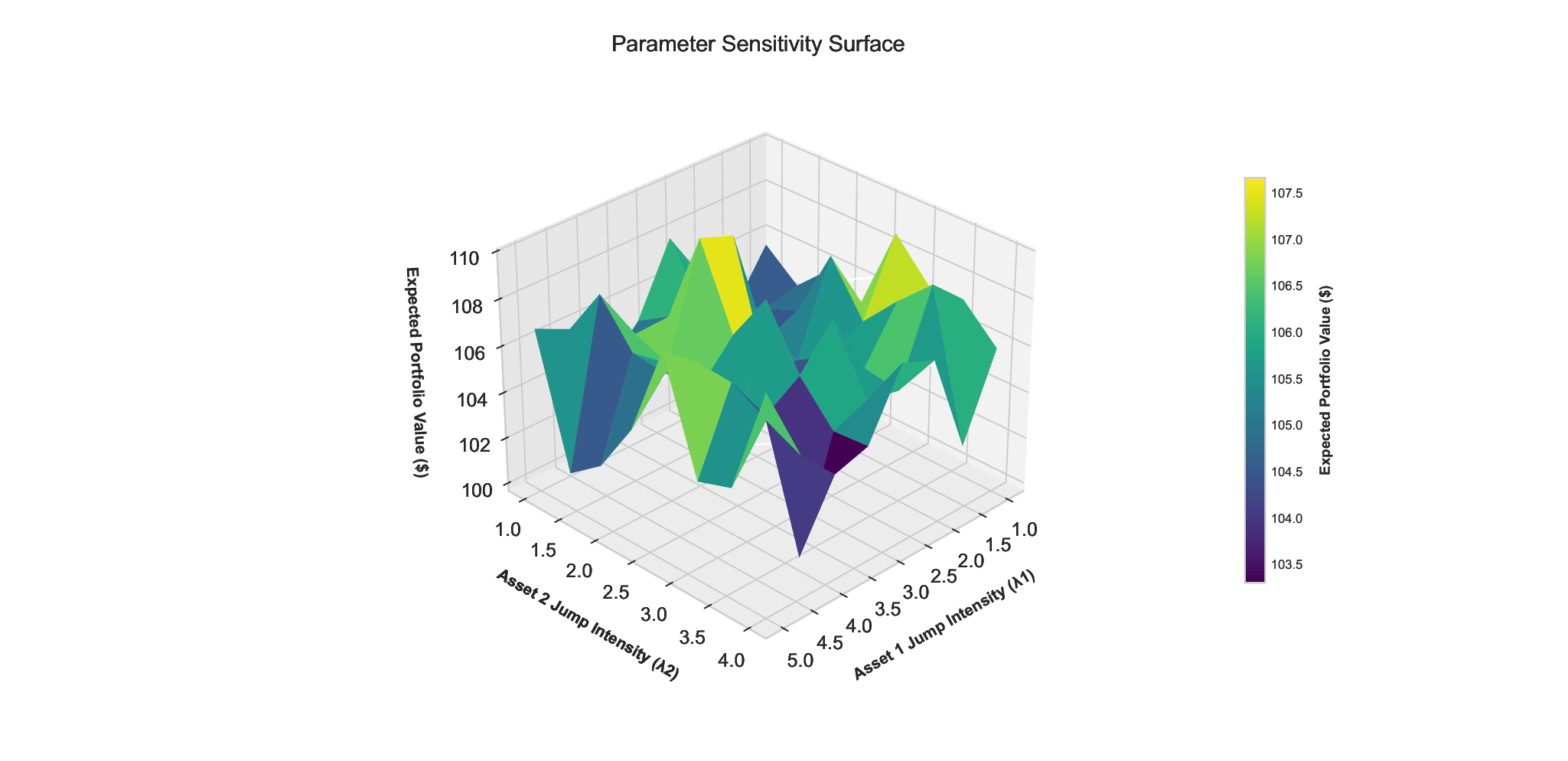}
\end{adjustbox}
             \caption{\scriptsize On top parameters set to  $\lambda = 1, T = 1, N = 2^{12}$  Example\hspace{1mm}\ref{EX3}.\\}
             {\scriptsize Asset1:$ X^x = 100$, jump sizes: Uniform$(-1/2, 1/2)$ -Asset2: $X^y =100$, jump sizes= Kou Model.}

 \label{EX3-Figure_1}
\end{figure}

\noindent The portfolio value at time $t$ is given by:
\begin{equation*}
P_t = w_1 X_t^x + w_2 X_t^y,
\end{equation*}

\noindent where $(w_1, w_2)$ are portfolio weights satisfying $w_1 + w_2 = 1$. We analyze three allocation strategies:
\begin{itemize}
    \item \textbf{Conservative}: $(w_1, w_2) = (0.2, 0.8)$
    \item \textbf{Balanced}: $(w_1, w_2) = (0.5, 0.5)$
    \item \textbf{Aggressive}: $(w_1, w_2) = (0.8, 0.2)$
\end{itemize}

\begin{figure}[H]

             \subfloat{\includegraphics[width=1.1\textwidth, keepaspectratio]{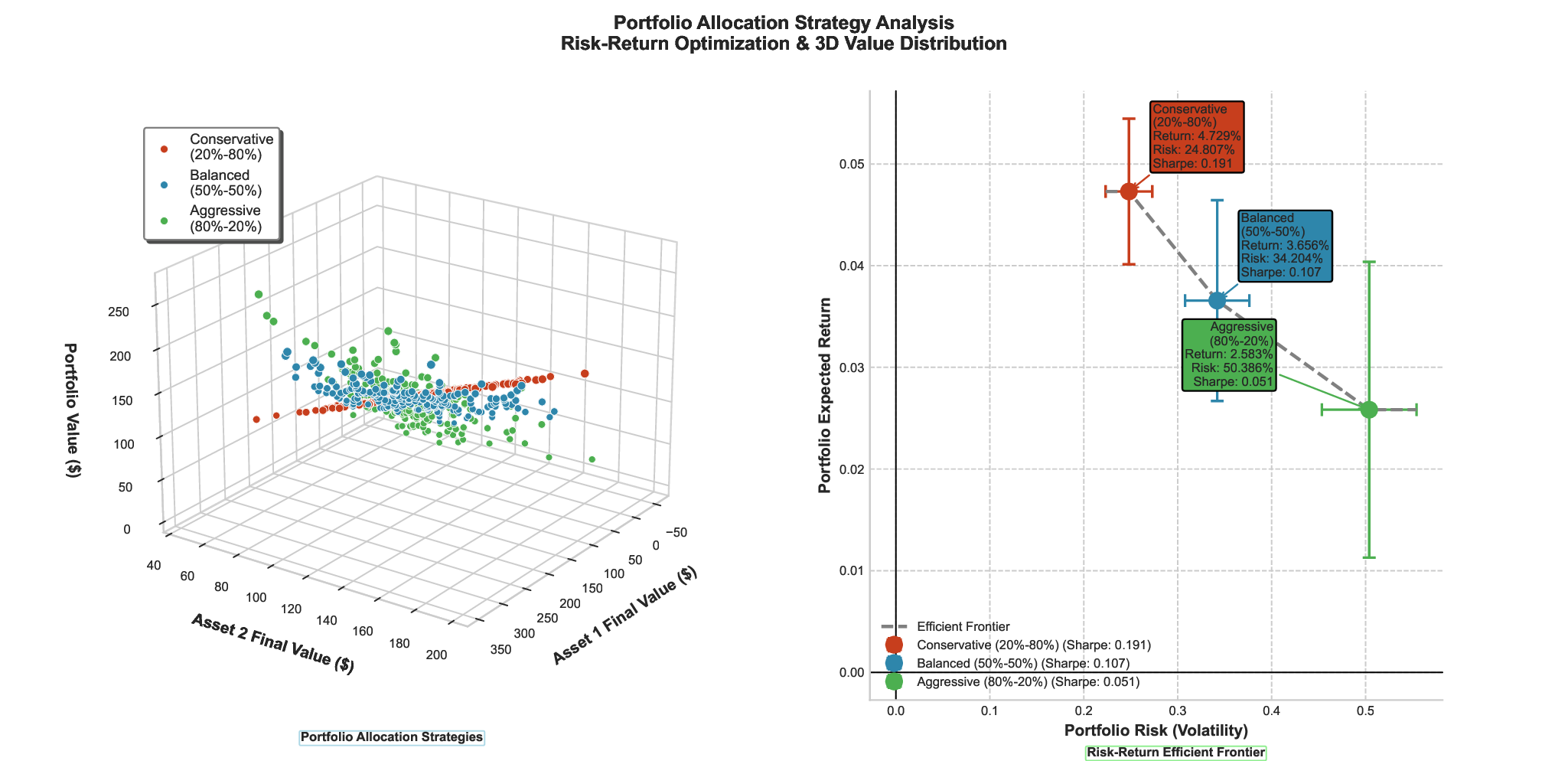}}

             \caption{\scriptsize Analyzing three allocation strategies.}
            
 \label{EX3-Figure_2}
\end{figure}
\noindent The mean-field component $\mathbb{E}[X_s^i]$ captures collective market behavior, creating interdependencies between individual asset dynamics and aggregate market states. Jump intensities $\lambda_i$ and magnitudes $\eta_{\text{up}}, \eta_{\text{down}}$ model sudden price movements due to market shocks or news arrivals.
\begin{figure}[H]

             \subfloat{\includegraphics[width=1\textwidth, keepaspectratio]{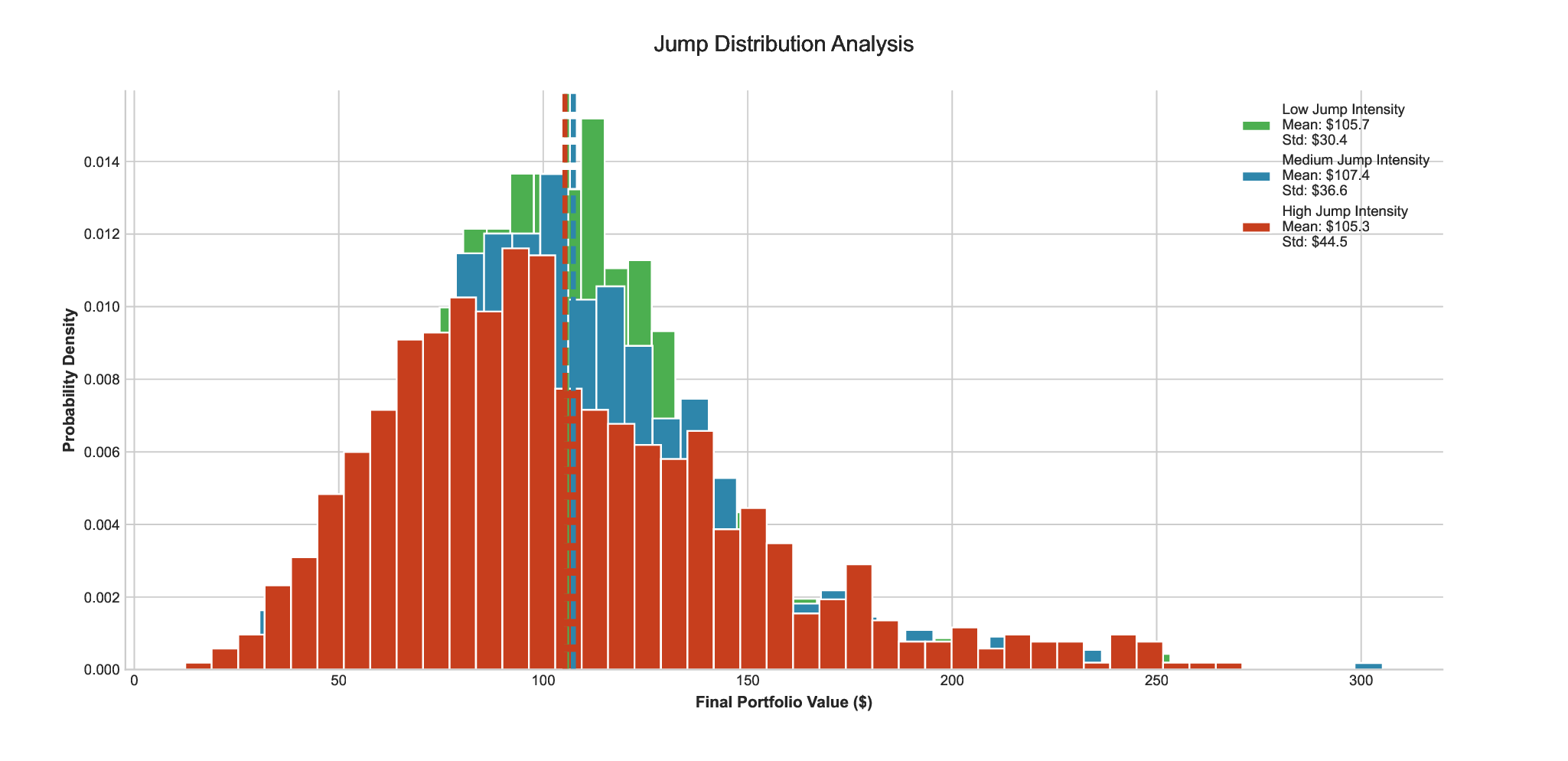}}

             \caption{\scriptsize On top parameters set to  $\lambda = 1, T = 1, N = 2^{12}$  Example\hspace{1mm}\ref{EX3}.\\}
             {\scriptsize Asset1:$ X^x = 100$, jump sizes: Uniform$(-1/2, 1/2)$ -Asset2: $X^y =100$, jump sizes= Kou Model.}
 \label{EX3-Figure_2}
\end{figure}

\begin{figure}[H]

             \subfloat{\includegraphics[width=1\textwidth, keepaspectratio]{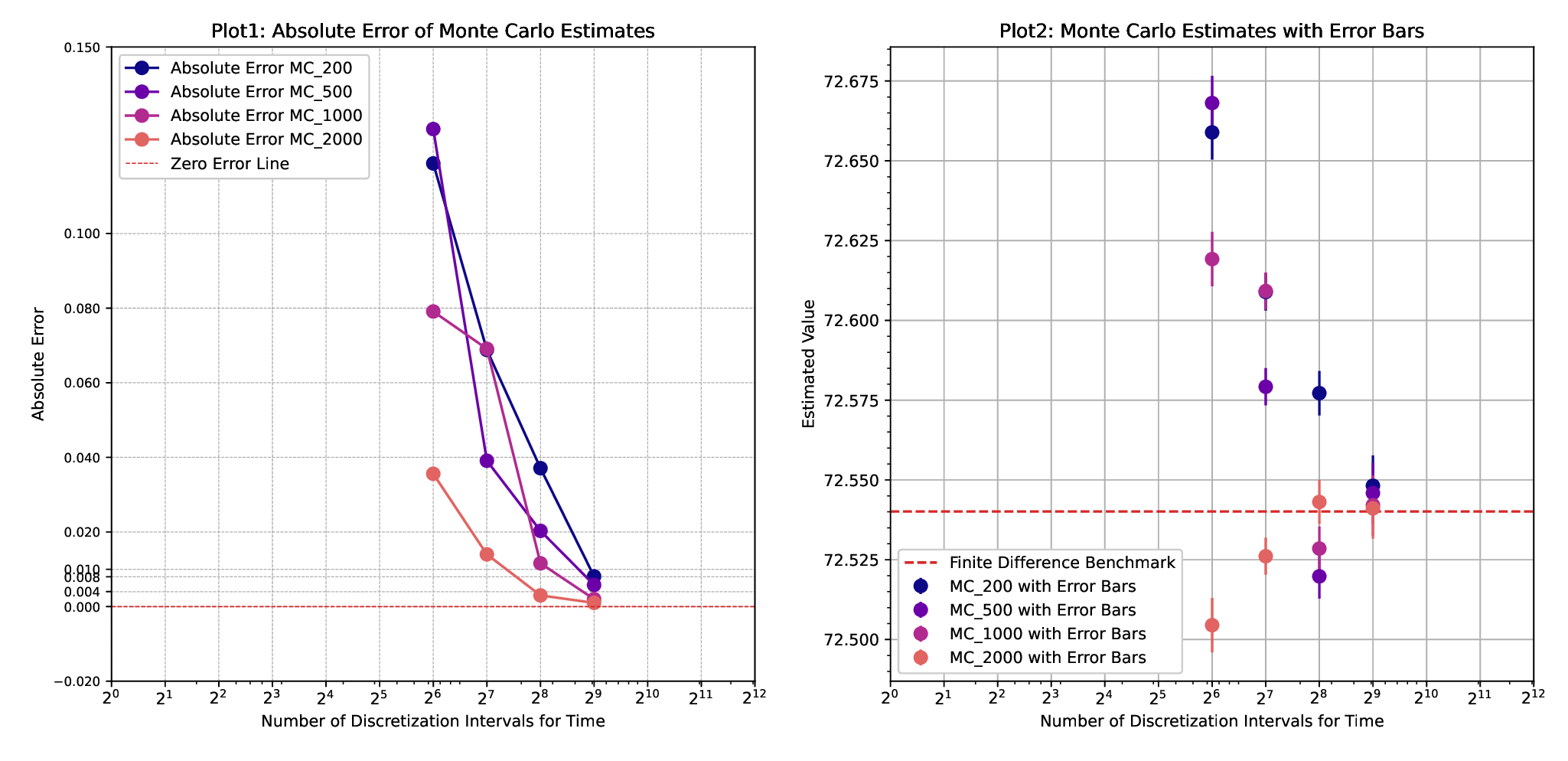}}
            \caption{\scriptsize Overview of Errors in Example\hspace{1mm}\ref{EX3}.\\}
             {\scriptsize Analysis of MMC methods: (Plot1) Absolute error showing reduction with increased discretization and the number of Monte Carlo simulations; (Plot2) MMC estimates with error bars, underscoring variability and the reliability of the methods used. }
 \label{EX3-Error}
\end{figure}

\begin{table}[H]
\caption{Evaluations of American Option Pricing in Context of a Portfolio Containing Two Assets Example \ref{EX3}.}
	 \begin{adjustbox}{max width=\textwidth}
		\begin{tabular}{ccccc|c}
			\toprule
			\textbf{Malliavin Monte-Carlo}	& \textbf{MC = 200}	& \textbf{MC = 500}     & \textbf{MC = 1000}     & \textbf{MC = 2000}     & \textbf{Finite Difference }\\
			\midrule
\multirow[m]{1}{*}{N = 64}	& 72.6589			& 72.6681			& 72.6192			& 72.5045      &72.5401\\
			  	                   
                   \midrule
\multirow[m]{1}{*}{N = 128}    & 72.6089			& 72.5792			& 72.6092			&72.5261     &72.5401\\

		\midrule	
\multirow[m]{1}{*}{N = 256}    & 72.5772			& 72.5198			& 72.5285			& 72.5431      &72.5401\\

		\midrule	  	                 
\multirow[m]{1}{*}{N = 512}    &72.5482			& 72.5459			& 72.5421			& 72.5411      &72.5401\\

			\bottomrule
		\end{tabular}
      \end{adjustbox}
	\noindent{\scriptsize{Datas: $X^x = 1, X^y = 10, K = 80$.}}
\end{table}
\end{example}
\appendix

\section{}\label{appendixA}
\subsection{Malliavin Derivative Operator}

\noindent The Malliavin calculus for jump processes presented in this section provides a rigorous framework for analyzing the differential properties of functionals involving both Wiener processes and Poisson processes, offering insights into their stochastic behavior and functional dependencies.\\
We begin by rigorously defining the core functional space that underpins Malliavin calculus in the Poisson spaces. The subsequent definitions and derived properties are essential building blocks for constructing the tailored Malliavin calculus framework we are developing here. \\
Define 
\( \mathcal{L}^2_p \) be the space of all predictable processes \( f : \Omega \times [0, 1] \times \Theta_0 \to \mathbb{R}^d \) with finite norm:
\begin{align}\label{LPnorm}
\|f\|_{\mathcal{L}^2_p}^p & :=\mathbb{E}\left[\int_0^1 \int_{\Theta_0} |f(s, z)|^2 \nu(dz) ds\right]^{p/2}+\mathbb{E}\left[\int_0^1 \int_{\Theta_0} |f(s, z)|^p \nu(dz) ds\right] < \infty,
\end{align}

\noindent \( \mathcal{H}^p \) be the space of all measurable adapted processes \( h : \Omega \times [0, 1] \to \mathbb{R}^d \) with finite norm:
\begin{align*}
\|h\|_{\mathcal{H}^p} & := \left[\mathbb{E}\left[\int_0^1 |h(s)|^2 ds\right]^{p/2}\right]^{1/p} < \infty, 
\end{align*}
and 
\( \mathcal{V}^p \) be the space of all predictable processes \( v_1 : \Omega \times [0, 1] \times \Theta_0 \to \mathbb{R}^d\) with finite norm:
\begin{align*}
\|v_1\|_{\mathcal{V}^p} & := \|\nabla_z v_1\|_{\mathcal{L}^2_p} + \|v_1\varrho\|_{\mathcal{L}^2_p} < \infty.
\end{align*}
Denote 
\begin{align*}
\mathcal{H}^{\infty -} & := \bigcap_{p \geq 1} \mathcal{H}^p, \quad \mathcal{V}^{\infty -} := \bigcap_{p \geq 1} \mathcal{V}^p.
\end{align*}
\noindent Let \( m \) be an integer. We define \( C_p^\infty(\mathbb{R}^m) \) as the set of all smooth functions on \( \mathbb{R}^m \) such that every derivative exhibits polynomial growth. Let \( \mathcal{F}C_p^\infty \) denote the class of Wiener-Poisson functionals on the probability space \( \Omega \), structured as follows:

\begin{align*}
F & = f(W(h_1), \ldots, W(h_{m_1}), N(g_1), \ldots, N(g_{m_2})),
\end{align*}
where \(f \in C_p^\infty(\mathbb{R}^{m_1+m_2})\), \(h_1, \ldots, h_{m_1} \in \mathcal{H}_0\) and \(g_1, \ldots, g_{m_2} \in \mathcal{V}_0\) are non-random and real-valued, and
\begin{align*}
W(h_i) & := \int_0^1 \langle h_i(s), dW_s \rangle_{\mathbb{R}^d}, \\
N(g_i) & := \int_0^1 \int_{\Theta_0} g_j(s, z) N(ds, dz).
\end{align*}
\noindent For any \(p > 1\) and \(\vartheta = (h, v_0) \in \mathcal{H}^p \times \mathcal{V}^p\), consider the following definition
\begin{align*}
D_\vartheta F & := \sum_{i=1}^{m_1} (\partial_i f)(\cdot) \int_0^1 \langle h(s), h_i(s) \rangle_{\mathbb{R}^d} ds + \notag \\
& \quad \sum_{j=1}^{m_2} (\partial_{j+m_1} f)(\cdot) \int_0^1 \int_{\Theta_0} \langle v_0(s, z), \nabla_z g_j(s, z) \rangle_{\mathbb{R}^d} N(ds, dz),
\end{align*}
where \((\cdot)\) stands for \(W(h_1), \ldots, W(h_{m_1}), N(g_1), \ldots, N(g_{m_2})\).
Denote $D^w F$ and $D_{r,z}F$ as the Malliavin derivative of F in the Wiener space and the Poisson space, respectively. In accordence with Chapters 4, and 11 in \cite{ref-nualart2018}, this notations confirm that $$D_\vartheta F =  \int_0^1 \langle h(s), D^w_sF \rangle_{\mathbb{R}^d} ds + \int_0^1 \int_{\Theta_0} \langle v(s, z), D_{s,z}F  \rangle_{\mathbb{R}^d} N(ds, dz).$$
\noindent \textbf{Definition 2.1}: For \(p > 1\) and \(\vartheta = (h, v) \in \mathcal{H}^p \times \mathcal{V}^p\), we define the first order Sobolev space \(W^{1, \vartheta, p}\) being the completion of \(\mathcal{F}C_p^\infty\) in \(\mathcal{L}^p(\Omega)\) with respect to the norm:
\begin{align*}
\|F\|_{\vartheta; 1, p} & := \|F\|_{\mathcal{L}^{p}} + \|D_\vartheta F\|_{\mathcal{L}^{p}}.
\end{align*}

\noindent The duality formula stated below (cf. \cite{ref-song2015}, Theorem 2.9).

\noindent 
\begin{theorem}\label{IBF}{\rm [Duality Formula]} Given \(\vartheta = (h, v) \in \mathcal{H}^{\infty-} \times \mathcal{V}^{\infty-}\) and \(p > 1\), for any \(F \in W^{1, \vartheta, p}\), we have
\begin{align*}
E[D_\vartheta F] & = -E[F \delta(\vartheta)],
\end{align*}
where
\begin{align*}
\delta(\vartheta) & := -\int_0^1 \langle h(s), dW_s \rangle + \int_0^1 \int_{\Theta_0} \frac{ \text{div}(\kappa v)( s, z)}{\kappa(z)} \tilde{N}(ds, dz),
\end{align*}
and \(\text{div}(\kappa v) := \sum_{i=1}^d \partial_{z_i} (\kappa v_i)\) stands for the divergence.
\end{theorem}

\noindent Given \(\vartheta = (h, v) \in \mathcal{H}^{\infty-} \times \mathcal{V}^{\infty-}\) and \(p > 1\), the operator \(D_\vartheta F\) on \(W^{1, \vartheta, p}\) in the weak topology of \(\mathcal{L}^p(\Omega)\) (\cite{ref-song2018}, Lemma  2.3), and according to Proposition 2.5 in \cite{ref-song2018} we have 
\begin{align*}
	D_\vartheta F & = 0 \quad \text{almost surely on } \{F = 0\}.
\end{align*}
In addition, the chain rule (\cite{ref-song2018}, Lemma 2.4) holds, i.e, 
for Lipschitz continuous function \(\phi = (\phi_1, \ldots, \phi_k) : \mathbb{R}^m \to \mathbb{R}^k\) and  \(F = (F_1, \ldots, F_m)\) with \(F_i \in W^{1, \vartheta, p}\), \(i = 1, 2, \ldots, m\), 
\begin{align*}
	D_\vartheta (\phi(F)) & = \nabla \phi(F) (D_\vartheta F_1, \ldots, D_\vartheta F_m) .
\end{align*}
Herein, we note some approximations needed to establish the conditional expectation. Approximate the indicator function \(\frac{1}{2} \mathbf{1}_{(-\epsilon, \epsilon)}\) by a sequence of smooth bounded functions \(\{\Phi_{\epsilon,n}\}_{n \geq 1}\) such that each \(\Phi_{\epsilon,n}\) has support contained in \((-\epsilon - \epsilon/n, \epsilon + \epsilon/n)\), equals 1 on \([-\epsilon + \epsilon/n, \epsilon - \epsilon/n]\), and connects values at the endpoints by linear segments on \((-\epsilon - \epsilon/n, -\epsilon + \epsilon/n)\) and \((\epsilon - \epsilon/n, \epsilon + \epsilon/n)\).  

\noindent Furthermore, we define the approximate cumulative function  
\begin{equation*}  
	\Psi_{\epsilon,n}(x) := \int_{-\infty}^x \Phi_{\epsilon,n}(y) \, dy + \epsilon c,
\end{equation*} 
which will be instrumental in employing the duality formula that links the Skorokhod integral with Malliavin derivatives. 
The sequence \((\Psi_{\epsilon,n})_{n}\) providing an approximation to the mollifier function \(\mathcal{H}_\epsilon : \mathbb{R} \to \mathbb{R}\) defined by  
\begin{equation*}  
	\mathcal{H}_\epsilon(x) =   
	\begin{cases}  
		\epsilon c, & x < -\epsilon, \\
		\frac{1}{2}(x + \epsilon) + \epsilon c, & -\epsilon \leq x < \epsilon, \\
		\epsilon + \epsilon c, & x \geq \epsilon,  
	\end{cases}  
\end{equation*}  
where \(c\) is a constant to be specified.  Obviously, $\frac{1}{\epsilon}  \mathcal{H}_\epsilon(x) \rightarrow \mathcal{H}(x)$ as $\epsilon \rightarrow 0$.

\section{Proof of Theorem \ref{th2}}\label{appendixB}
\noindent First, let us consider a Borel measurable function \( f \in C'_b \), such that we can apply Theorem \eqref{th1}. Due to the properties of conditional expectations and relation \eqref{condition2}, we have in representation \eqref{eq7} that
\begin{align*}
\mathbb{E} \Big[ f'(F) &\mathcal{H}(G - a) \int_{0}^{T}\int_{\Theta_0}^{}\langle u(t, z), D_{t, z} F \rangle_{\mathbb{R}^d}  \,{N}(dt, dz) \Big] \notag \\
&= \mathbb{E} \left[ \mathbb{E} \left[ f'(F) \mathcal{H}(G - a) \int_{0}^{T}\int_{\Theta_0}^{}\langle u(t, z), D_{t, z} F \rangle_{\mathbb{R}^d} \, {N}(dt, dz) \bigg| \sigma(F, G) \right] \right] = 0.
\end{align*}
Thus, we obtain representation \eqref{eqt9}.
\noindent Now, consider a Borel measurable function \( f \) for which \( f(F) \in \mathcal{L}^2(\Omega) \). Let \( C_K^{\infty}(\mathbb{R}) \) and \( C_K^{\infty}(\Omega) \) be the spaces of infinitely differentiable functions with compact support, defined respectively in \( \mathbb{R} \) and \( \Omega \). Since \( C_K^{\infty}(\Omega) \) is dense in \( \mathcal{L}^2(\Omega) \), there exists a sequence of functions \( f_n \in C_K^{\infty}(\mathbb{R}) \), such that \( f_n(F) \) converges to \( f(F) \) in \(\mathcal{L}^2(\Omega) \). In part i), we concluded that for any function \( f_n \) in this sequence, representation \eqref{eqt9} holds. By convergence arguments, we conclude that expression \eqref{eqt9} also holds for the limit function \( f \) as follows. For any \( f_n \), we denote
\begin{align*}
g_n(\alpha) := \mathbb{E}[f_n(F) | G = \alpha] = \frac{\mathbb{E}[ \mathcal{H}(G - a) f_n(F)\delta(u)]}{\mathbb{E}[\mathcal{H}(G - a) \delta(u)]}.
\end{align*}
In addition, we define
\begin{align*}
g(\alpha) & := \frac{\mathbb{E}[\mathcal{H}(G - \alpha) \delta(u)f(F)]}{\mathbb{E}[\mathcal{H}(G - \alpha) \delta(u)]}.
\end{align*}
Using the Cauchy–Schwarz inequality, we derive that
\begin{align*}
|g(\alpha) - g_n(\alpha)| &\leq \frac{\mathbb{E} \left[ |f(F) - f_n(F)| |\mathcal{H}(G - \alpha) \delta(u)| \right]}{\mathbb{E}[\mathcal{H}(G - \alpha) \delta(u)]} \notag \\
& \leq \frac{\mathbb{E}[|f(F) - f_n(F)|^2]^{1/2} \mathbb{E}[|\mathcal{H}(G - \alpha) \delta(u)|^2]^{1/2}}{\mathbb{E}[\mathcal{H}(G - \alpha) \delta(u)]}.
\end{align*}
For any \( \alpha \in \mathbb{R} \), we have
\begin{align*}
\frac{\mathbb{E}[|\mathcal{H}(G - \alpha) \delta(u)|^2]^{1/2}}{\mathbb{E}[\mathcal{H}(G - \alpha) \delta(u)]} &< \infty.
\end{align*}
By the density argument, \( f_n(F) \) converges to \( f(F) \) in the \(\mathcal{L}^2\)-sense, hence we obtain that
\begin{align*}
|g(\alpha) - g_n(\alpha)| & \to 0 \text{ as } n \to \infty, \text{ for all } \alpha \in \mathbb{R}.
\end{align*}
Thus, \( g_n(\alpha) \) converges to \( g(\alpha) \). Moreover, \( g_n(\alpha) \) converges to \( \mathbb{E}[f(F) | G = \alpha] \) by the conditional dominated convergence theorem. Therefore, we conclude that \( g(\alpha) \) equals the latter conditional expectation.


\section{Algorithm of Pricing American Options}\label{appendixC}
\noindent Consider an American option expiring at $T$, which confers the right to exercise at any time within $[t,T]$. We model the underlying asset by a mean-field SDE with jump $X$ on $\mathbb{R}^d$ and write $\Phi(X_s)$ for the payoff stream. The arbitrage-free price at time $t$ is characterized by the optimal stopping problem
\begin{align*}
\mathcal{P}(t,x)
= \sup_{\tau \in \mathcal{T}_{t,T}}
\mathbb{E}_{t,x}\!\left[
  e^{-\int_t^{\tau} r_s\,\mathrm{d}s}\,\Phi(X_\tau)
\right],
\end{align*}
\noindent where $\mathcal{T}_{t,T}$ is the set of stopping times taking values in $[t,T]$, and $r$ denotes the deterministic short rate. By appealing to the Snell-envelope framework, the earliest optimal exercise time is identified as
\begin{align*}
\tau_t^{*}=\inf\{\, s \in [t,T] : \mathcal{P}(s,X_s)=\Phi(X_s) \,\}.
\end{align*}
The pricing function $\mathcal{P}(t,x)$ then satisfies the variational characterization: it solves
\begin{align*}
\big(\partial_t + \mathscr{L}\big)\mathcal{P}(t,x) - r\,\mathcal{P}(t,x) = 0,
\end{align*}
whenever $\mathcal{P}(t,x)$ exceeds the payoff $\Phi(x)$, and matches the terminal condition $\mathcal{P}(T,x)=\Phi(x)$. The operator $\mathscr{L}$ is the infinitesimal generator of $X$. Delivering a rigorous statement is nontrivial. The function $\mathcal{P}$ generally falls short of classical regularity ($C^{1}$ in time, $C^{2}$ in space). Its behavior in the vicinity of the free boundary $\{(t,x): \mathcal{P}(t,x)=\Phi(x)\}$ is particularly delicate, effectively injecting an extra term into the PDE. These features necessitate weak frameworks: the variational-inequality formulation of Bensoussan--Lions \citep{ref-bensoussan1984}, the viscosity formulation of El Karoui \citep{ref-el1997}, and the Sobolev formulation of Bally \citep{ref-bally2002}.\\
\noindent In practice, one computes the time-zero price $\mathcal{P}(0,x)$ via a Bellman dynamic programming scheme. Construct a time grid $0=t_0<t_1<\cdots<t_n=T$ over $[0,T]$ with uniform step $\varepsilon=T/n$, and approximate the process $(X_t)_{t\in[0,T]}$ by the discrete sequence $(\bar X_{t_k})_{k=0,\ldots,n}$, where $\bar X_{t_k} \approx X_{t_k}$. The value $\mathcal{P}(t_k,\bar X_{t_k})$ is then estimated by a sequence $\bar {\mathcal{P}}_{t_k}(\bar X_{t_k})$, defined recursively through the ensuing relation, see also Bally et al.\citep{ref-bally2005}.\\

\begin{theorem}
Given $\varepsilon = T/n \in (0,1)$, define
$\bar{\mathcal{P}}_{t_n}(\bar X_{t_n}) = \Phi(\bar X_{n\Delta t})$,
and for any $k = n-1, n-2, \ldots, 1, 0$,
\begin{align*}
\bar{\mathcal{P}}_{t_k}(\bar X_{t_k})
= \max\!\left\{
\Phi(\bar X_{t_k}),
\, e^{-r\varepsilon}\,
\mathbb{E}\!\left[
\bar{\mathcal{P}}_{t_{(k+1)}}\!\big(\bar X_{t_{(k+1)}}\big)
\;\middle|\; \bar X_{t_k}= \alpha
\right]
\right\}.
\end{align*}
Then $\bar{\mathcal{P}}_{t_k}(\bar X_{t_k}) \simeq \mathcal{P}\big(t_k, X_{t_k} \big)$.
\end{theorem}

\noindent Our computation of 
$\mathbb{E}\!\left[
\bar{\mathcal{P}}_{t_{(k+1)}}\!\big(\bar X_{t_{(k+1)}}\big)
\;\middle|\; \bar X_{t_k}= \alpha \right],$
relies on Malliavin calculus, leveraging the identities in \eqref{vareq} and a Monte Carlo scheme.\\
%
\noindent Now, we simulate the wealth process $X_t^x$ using the Euler discretization in equation \eqref{eqt3}, generating $M$ sample trajectories as the algorithm’s core input. For convergence results of the Euler method for mean-field SDEs with jumps, see \citep{ref-sun2021}. Time is discretized on $[0, T]$ using a uniform partition with step size $\Delta t$, yielding $N =T / \Delta t $ subintervals. We employ a small time step to ensure that the numerical scheme accurately captures the dynamics of the mean-field SDE with jumps. The pricing procedure based on the Malliavin Monte Carlo (MMC) method can be outlined as follows. To benchmark our approach, we apply a finite-difference discretization to the test problems and juxtapose its outputs with those of our method. Because an exhaustive discussion of finite differences is outside the remit of this paper, we provide only a succinct summary. 
\textbf{Pricing Algorithm by MMC Method (Step-by-Step)}
\begin{enumerate}
  \item To simulate \( X_{t_k}^{\,n,x} \) using the Euler-scheme approximation of the solution, consider any integer \( N \in \mathbb{N} \) and the partition of the interval \([0, T]\) defined by \( t_k = k \varepsilon \), in this context \( \varepsilon:=\Delta t= \frac{T}{N} \) and \( k = 0, \ldots, N \), as illustrated below
  \begin{equation*}
  \begin{aligned}
  X_{t_{k+1}}^{n,x} ={}& X_{t_k}^{n,x}
      + b\bigl(t_k, X_{t_k}^{n,x}, \rho_{t_k}\bigr)\, \Delta t
      + \sum_{q=1}^k \sigma_m\bigl(t_k, X_{t_k}^{n,x}, \pi_{t_k}\bigr)\, \Delta W_k^{\,q}+ \sum_{j=1}^{\mathcal{N}_k} \lambda\bigl(t_k, X_{t_k}^{n,x}, z_j, \eta_{t_k}\bigr),
  \end{aligned}
  \end{equation*}
  \noindent where $\Delta W_k^{\,q} = W_{t_{k+1}}^{\,q} - W_{t_k}^{\,q}$ are Brownian increments, $\mathcal{N}_k \sim \mathrm{Poisson}(\nu\,\Delta t)$ is the number of jumps on $[t_k,t_{k+1}]$, and $z_j$ is the size of the $j$-th jump.

  \item For all $m = 1, ..., M$, when $k = N$ we have
  \begin{equation*}
  \bar{\mathcal{P}}_{t_N}\!\left(X_{t_N} ^m\right)
  = \Phi\!\left(X_{t_N} ^m\right) \equiv \Phi(X_T).
  \end{equation*}

\item At each time point $t_k$ corresponding to $k = N-1, N-2, \ldots, 1$, we iteratively estimate the conditional expectations using the Malliavin Monte Carlo technique, considering each index $m$ in the set $\{1, 2, \ldots, M\}$. Assuming $\alpha = X_{t_k} ^m$, we arrive at the following implications:
\begin{align*}
  \bar{\mathcal{P}}_{t_k}\!\left(X_{t_k} ^m\right)
=\max\!\Biggl\{
    \Phi\!\left(X_{t_k} ^m\right),\;
    e^{-r\varepsilon}\,
    \underbrace{\mathbb{E}\!\left[
      \bar{\mathcal{P}}_{t_{k+1}}\!\left(X_{t_{k+1}}\right)
      \,\middle|\, X_{t_k}= X_{t_k} ^m
    \right]}_{\clubsuit}  \Biggr\}.
\end{align*}
\noindent where
\begin{align}\label{MMC}
&\mathbb{E}\!\big[  \bar{\mathcal{P}}_{t_{k + 1}}\!(X_{t_{k + 1}})\,|\,X_{t_k} = X_{t_k}^m \big] \notag\\
&= \frac{\mathbb{E}\left[\bar{\mathcal{P}}_{t_{k + 1}}\!(X_{t_{k + 1}})\Big(\psi(X_{t_k} - X_{t_k} ^m) + \Pi\big[\mathcal H(X_{t_k} - X_{t_k} ^m) - \Psi(X_{t_k} - X_{t_k} ^m)\big]\Big) \right]}{\mathbb{E}\Big[\psi(X_{t_k} - X_{t_k} ^m) + \Pi\big[\mathcal H(X_{t_k} - X_{t_k} ^m) - \Psi(X_{t_k} - X_{t_k} ^m)\big] \Big]} \notag\\
&= \frac{\sum\limits_{l=1}^{M} \left[\bar{\mathcal{P}}_{t_{k + 1}}\!(X_{t_{k + 1}}^l)\Big(\psi(X_{t_k} ^l - X_{t_k} ^m) + \Pi^l\big[\mathcal H(X_{t_k} ^l - X_{t_k} ^m) - \Psi(X_{t_k} ^l - X_{t_k} ^m)\big]\Big) \right]}{\sum\limits_{l=1}^{M} \Big[\psi(X_{t_k} ^l - X_{t_k} ^m) + \Pi^l\big[\mathcal H(X_{t_k} ^l - X_{t_k} ^m) - \Psi(X_{t_k} ^l - X_{t_k} ^m)\big] \Big]}.
\end{align}
 \noindent The value of $\Pi^l$ is determined based on the expression provided in \eqref{Pi}. To sort the wealth at time \( t_n \) in a specified manner, such as \( X_{t_n}^1 \leq X_{t_n}^2 \leq \cdots \leq X_{t_n}^M \), we can adopt an iterative approach as follows. Starting with the \( M \)-th path, the summations initially consist of a single term associated with this path. Once the \( M \)-th path is addressed, we proceed to the \( (M-1) \)-th path, where the sums incorporate two terms: one for the \( M \)-th path and another for the \( (M-1) \)-th path. Given that we have already determined the solution for the \( M \)-th path, we can leverage this information to solve for the \( (M-1) \)-th path in a similar manner. This iterative procedure continues until all necessary path solutions are obtained.  This guarantees a significant reduction in computational complexity, boosting the overall efficiency of the implementation. Ultimately, the samples \( \left( \bar{\mathcal{P}}_\varepsilon(X_\varepsilon^m)\right)_{m=1,\ldots,M} \) are available.

  \item Determination of the price through the following relationship
\begin{align*}
\bar{\mathcal{P}}_0(x) &= \max \left( \Phi(x), e^{-r \varepsilon} \frac{1}{M} \sum_{m=1}^{M} \bar{\mathcal{P}}_\varepsilon(X_\varepsilon^m) \right).
\end{align*}
\end{enumerate}
\noindent Finally, it is important to highlight that a localization function has been employed to mitigate variance in the implementation of this method. In this context, relationships \eqref{si} and \eqref{lambdaa} have been utilized in the computation of \eqref{MMC}.



%





\end{document}